\documentclass{siamart250211}
\usepackage{a4wide,amsmath,amssymb}
\usepackage[toc,page]{appendix}
\usepackage{hyperref}
\usepackage{url}
\usepackage{mathtools}
\usepackage{xcolor}
\usepackage{bbm}

\newtheorem{thm}{Theorem}[section]
\newtheorem{prop}[thm]{Proposition}
\newtheorem{cor}[thm]{Corollary}
\theoremstyle{remark}
\newtheorem{rem}[thm]{Remark}

\theoremstyle{definition}
\newtheorem{defn}[thm]{Definition}
\theoremstyle{Claim}

\newcommand{\lra}{\longrightarrow}
\newcommand{\NN}{{\mathbb{N}}}

\newcommand{\EE}{\mathbb{E}}
\newcommand{\PP}{\mathbb{P}}
\newcommand{\R}{\mathbb{R}}
\newcommand{\Rd}{\mathbb{R}^d}

\newcommand{\mB}{\mathcal B}
\newcommand{\mV}{\mathcal V}
\newcommand{\mI}{\mathcal I}
\newcommand{\mK}{\mathcal K}

\newcommand{\mT}{\mathcal T}
\newcommand{\mP}{\mathcal P}
\newcommand{\mF}{\mathcal F}
\newcommand{\mM}{\mathcal M}
\newcommand{\mW}{\mathcal W}
\newcommand{\X}{\mathcal{X}}
\newcommand{\cE}{\mathcal{E}}

\allowdisplaybreaks
\newcommand{\eps}{\varepsilon}

\newcommand{\half}{\frac 1 2}

\newcommand{\m}{\mu}

\DeclareMathOperator{\diver}{div}
\DeclareMathOperator{\supp}{supp}

\theoremstyle{plain}

\def\sideremark#1{\ifvmode\leavevmode\fi\vadjust{
		\vbox to0pt{\hbox to 0pt{\hskip\hsize\hskip1em
				\vbox{\hsize3cm\tiny\raggedright\pretolerance10000
					\noindent #1\hfill}\hss}\vbox to8pt{\vfil}\vss}}}

\newcommand{\D}{\Delta}
\newcommand{\Dx}{{\Delta x}}
\newcommand{\Ver}{{\mV}^\Dx}

\newcommand{\tr}{\on{tr}}

\newcommand{\mbb}{\mathbb}

\newcommand{\on}{\operatorname}

\makeatletter
\@namedef{subjclassname@2020}{\textup{2020} Mathematics Subject Classification}
\makeatother
\numberwithin{equation}{section}

\title{Approximation of  viscous transport and conservative equations with one sided Lipschitz velocity fields\thanks{Submitted to the editors DATE.
\funding{This work was partially supported by by  INdAM-GNAMPA, PRIN PNRR P20225SP98
 	``Some mathematical approaches to climate change and its impacts'', ``INdAM--GNCS Project'', codice CUPE53C24001950001.}}}

\author{Fabio Camilli\thanks{Dip. di Ingegneria e Geologia, Univ. ``G. D'Annunzio'' Chieti-Pescara,
			viale Pindaro 42, 65127 Pescara (Italy),
  (\email{fabio.camilli@unich.it}).}
\and Adriano Festa \thanks{Dip. di Scienze Matematiche ``G. L. Lagrange'', Politecnico di Torino, Corso Duca degli Abruzzi  24, 10129 Torino  (Italy),
  (\email{adriano.festa@polito.it}).}
\and Luciano Marzufero\thanks{Facoltà di Economia e Management, Libera Università di Bolzano, piazza Università  1,
		39100 Bolzano   (Italy)}, (\email{luciano.marzufero@unibz.it}).}

\begin{document}

	\maketitle

\begin{abstract}
The aim of this work is to investigate semi-Lagrangian  approximation schemes on unstructured grids for   viscous transport  and   conservative equations with measurable coefficients that satisfy a one-sided Lipschitz condition.  To establish the convergence of the schemes, we exploit the  characterization of the solution for these equations   expressed in terms of measurable time-dependent viscosity solution  and, respectively, duality solution. 	We supplement our theoretical analysis with various numerical examples	 to illustrate the features of the schemes.
\end{abstract}



\section{Introduction}
We investigate semi-Lagrangian approximation schemes for the viscous transport equation  
\begin{equation}\label{eq:HJ}
	-\frac{\partial u}{\partial t} - \tr[a(t,x)D^2 u] + b(t,x) \cdot D u = 0 \quad \text{in } (0,T) \times \mathbb{R}^d, \quad \text{with } u(T,\cdot) = u_T,
\end{equation}
and the viscous conservative equation  
\begin{equation}\label{eq:FP}
	\frac{\partial f}{\partial t} -  D^2\cdot [a(t,x)f] -\diver(  b(t,x) f )  = 0 \quad \text{in } (0,T) \times \mathbb{R}^d, \quad \text{with } f(0,\cdot) = f_0.
\end{equation}
The well-posedness of these problems has been extensively studied in the literature. For the first-order case where $a(x) \equiv 0$, relevant results can be found in \cite{am, dpl}. The second-order case has been investigated in \cite{lbl, f}, and a more comprehensive list of references is available in \cite{ls}. These works generally assume that the divergence of $b$ is absolutely continuous with respect to the Lebesgue measure.

In this work, we consider $b: [0,T] \times \mathbb{R}^d \lra \mathbb{R}^d$ as a measurable function satisfying the one-sided Lipschitz condition ((OSLC) in short)  
\begin{equation}
(b(t,x) - b(t,y)) \cdot (x-y) \geq -C(t)|x-y|^2, \quad \forall t \in [0,T], \, x,y \in \mathbb{R}^d, \label{OSLC}
\end{equation}
for some $C \in L^1(0,T)$ with $C \geq 0$. Additionally, $a$ is assumed to be a regular,  degenerate matrix. Since the (OSLC) framework does not rule out the possibility that the measure \(\diver b\)  may include a singular component, this limitation prevents   the application of the aforementioned theories to the present situation. 

For the first-order case satisfying the (OSLC) condition, the problems \eqref{eq:HJ} and \eqref{eq:FP} have been addressed for example in \cite{bj, bjm, cs, pr} (see Remark \ref{rem:first_order} concerning the sign of $b$). In \cite{ls},  a general theory  has been recently developed both in the first order case and in the second order one. In this latter work, various characterizations of the  ``good" solutions, i.e. solutions obtained as limits of equations with regular approximations of $ b $,  are provided. Specifically, the  ``good" solutions are characterized as measurable time-dependent viscosity solutions   for the transport equation, and duality solutions for the conservative equation.  

Regarding the approximation of these equations in the regular case, a vast body of literature exists due to their relevance in diverse fields such as fluid dynamics, financial mathematics, optimal transport, and Mean-Field Games. For first-order equations satisfying the (OSLC) condition, we refer to \cite{dl, dlv, gj, CCF}. The second-order case in the Di Perna-Lions setting has been recently studied in \cite{ns}, whereas, to the best of our knowledge, no articles have addressed the (OSLC) case.

Here, we examine semi-Lagrangian approximation schemes for \eqref{eq:HJ} and \eqref{eq:FP} under the (OSLC) condition. For an introduction to semi-Lagrangian scheme for the first-order case, we refer to  \cite{ff}, while the second-order case has been studied, for instance, in \cite{bf, cf, cs1, cs2}.  
Semi-Lagrangian methods offer a distinct advantage over traditional approaches like finite differences and finite volumes. Instead of directly approximating the partial differential equations themselves, they focus on approximating the underlying dynamics. This allows  to employ the assumption (OSLC) to derive properties of the discrete trajectories, which subsequently inform the approximation scheme. These properties enable us to establish the convergence of the semi-Lagrangian approximation scheme for the viscous transport equation by exploiting the characterization of the solution as time-measurable viscosity solution. The proof utilizes a classical stability result by Barles and Souganidis \cite{bs}, adapted to this specific context.  

For the conservation equation, we adopt the scheme proposed in \cite{cs1, cs2}, which is designed to handle linear and nonlinear Fokker-Planck equations with continuous coefficients on structured grids. For linear equations, our work generalizes their results  to encompass cases involving irregular coefficients and general unstructured grids. Moreover, we precisely characterize the scheme's limit as a duality solution of \eqref{eq:FP}. These generalizations present new challenges, particularly in establishing the regularity properties of the discrete solution, which are critical for achieving convergence in the limit.  At the discrete level, we derive a  duality formula analogous to the one in the continuous setting and, exploiting this formula and the convergence result for \eqref{eq:HJ}, we prove that the scheme converges to the duality solution of \eqref{eq:FP} in the sense of weak convergence of measures.\\
While the results presented in this paper apply to the first-order problem, it is important to note that, in this case, error estimates for a finite difference scheme are derived in terms of the \( L^\infty([0,T], L^1(\mathbb{R}^d)) \) norm for the transport equation (see \cite{dl}) and the \( p \)-Wasserstein distance for conservative equation (see \cite{dlv}).\\
In the numerical section, we show in practice the performances of the numerical techniques proposed here. We consider an example in a one-dimensional case, where it is possible to compare the approximation with the analytical solution of the problem  and two examples  in the two dimensional  case, the first one with a continuous vector field, and the second one with a discontinuous vector field   verifying the (OSLC) condition.
\subsection*{Notations}
Let $ C_b(\mathbb{R}^d) $ represents the set of bounded continuous functions   and $C^{0,1}(\R^d)$ the set of the Lipschitz continuous function in $\R^d$. 
We denote by $ \mM(\mathbb{R}^d) $ the space of non-negative measures on $ \mathbb{R}^d $, defined with respect to the Borel $ \sigma $-algebra $ \mathcal{B}(\mathbb{R}^d) $. This space is equipped with the weak topology $ \sigma(\mM(\mathbb{R}^d), C_b(\mathbb{R}^d)) $.  Next, we consider the space $ C([0, T]; \mM(\mathbb{R}^d)) $, which represents flows of measures. This space is endowed with the topology of uniform convergence over $[0, T]$. \\
Let $\mP(\R^d)$ be the subset of ${\mathcal M}(\R^d)$ consisting of probability measures defined on $(\R^d, \mathcal{B}(\R^d))$. Furthermore, we define $\mP_p(\R^d)$ as the space of probability measures that have a finite $p$-th order moment for $p \geq 1$, given by:  
$$  
\mP_p(\R^d) = \left\{ \mu \in \mP(\R^d) : \int_{\R^d} |x|^p \, \mu(dx) < \infty \right\}.  
$$  
This space is equipped with the standard Wasserstein distance $\mW_p$ (see \cite{vil}). For $p=1$, by the  Kantorovich–Rubinstein duality formula, we have for $\mu,\mu'\in\mP_1(\R^d)$
\[
\mW_1(\m,\m')=\sup\left\{\int_{\R^d} \phi(x)d(\m-\m'):\,\phi:\R^d\lra\R\,\mbox{is $1$-Lipschitz continuous }\right\}.
\] 
For a locally bounded function $g:\R^d\to \R$, we define
\begin{align*}
&\liminf_{z\to x} \hskip -2pt_{*}\, g(z)	=\lim_{r\to 0}\,\inf\{g(z):\,z\in B(0,r)\},\\
&\limsup_{z\to x} \hskip -2pt^{*}\, g(z)	=\lim_{r\to 0}\,\sup\{g(z):\,z\in B(0,r)\}.
\end{align*}
\section{The model problem}
The results of this section are taken from \cite{ls}, whose notation we will adhere to. 
We assume that $b:[0,T]\times\R^d\lra\R^d$ is a measurable function satisfying: for all $t \in [0,T]$ and $x,y \in \R^d$,
\begin{equation}\label{hyp:OSLC}
	\begin{split}
		&\quad \sup_{x \in \R^d} \frac{|b(t,x)|}{1 + |x|} \le C_0(t),   \\[4pt]
		&\quad (b(t,x) - b(t,y)) \cdot (x-y) \ge -C_1(t)|x-y|^2	,
	\end{split}
\end{equation}
for some non negative functions $C_0,C_1 \in L^1(0,T)$. Moreover  $a = \frac{1}{2} \sigma \sigma^T$ for $\sigma: [0,T] \times \R^d \lra \R^{d \times r}$ satisfying
\begin{equation}\label{hyp:sigma}
	\sup_{x \in \R^d} \frac{ |\sigma(t,x)|}{1 + |x|} + \sup_{y,z \in \R^d} \frac{|\sigma(t,y) - \sigma(t,z)|}{|y-z|} \le C_2(t)
\end{equation}
for a non negative function $ C_2\in L^2(0,T)$.
We consider the   SDE
\begin{equation}\label{eq:SDE}
	\begin{dcases}
		d_s \Phi_{s,t}(x) = -b(s,\Phi_{s,t}(x))ds + \sigma(s,\Phi_{s,t}(x)) dW_s  & s \in [t,T],\\
		\Phi_{t,t}(x) = x\in\R^d,
	\end{dcases}
\end{equation}
where $W: \Omega \times [0,T] \lra \R^r$ is a standard Brownian motion on a given probability space $(\Omega, \mF, \mbb P)$.
\begin{prop}\label{prop:stoch_flow}
	For every $(t,x) \in [0,T] \times \R^d$ and $\mathbb{P}$-almost surely, there exists a unique strong solution $\Phi_{s,t}(x)$ of \eqref{eq:SDE} defined on $[t,T] \times \R^d$. For all $p \in [2,\infty)$, there exists a constant $C = C_{p} > 0$ depending only on the assumptions \eqref{hyp:OSLC} and \eqref{hyp:sigma} such that
	\begin{equation}\label{eq:Lip_traj}
		\mbb E |\Phi_{s,t}(x) - \Phi_{s,t}(y)|^p \le C |x-y|^p \quad \text{for all } 0 \le t \le s \le T \text{ and } x,y \in \R^d,
	\end{equation}
	\begin{equation}\label{eq:bd_traj}
		\EE|\Phi_{s,t}(x)|^p \le C(|x|^p + 1) \quad \text{for all } 0 \le t \le s \le T \text{ and } x \in \R^d,
	\end{equation}
	and
	\begin{equation*} 
			\EE |\Phi_{s_1,t}(x) - \Phi_{s_2,t}(x)|^p  \le C (1 + |x|)|s_1 - s_2|^{p/2}
	\end{equation*}
   for all $t \in [0,T]$, $s_1,s_2 \in [t,T]$, and $ x \in \R^d$. Moreover, for all $0 \le r \le s \le t \le T$, we have $\Phi_{t,s} \circ \Phi_{s,r}= \Phi_{t,r}$  with probability one.
\par
\smallskip
	If $(b^\eps)_{\eps > 0}$ is a  regularization of $b$ satisfying 
	\begin{equation}\label{eq:breg}
		\left\{
		\begin{split}
			&(b^\eps)_{\eps > 0} \subset L^1([0,T], C^{0,1}(\R^d)), \quad \lim_{\eps \to 0} b^\eps = b \text{ a.e. in } [0,T] \times \R^d, \text{ and}\\
			&b^\eps \text{ satisfies \eqref{hyp:OSLC} uniformly in $\eps > 0$},
		\end{split}
		\right.
	\end{equation}
 then  the corresponding stochastic flows $\Phi^\eps$ converge locally uniformly to $\Phi$ as $\eps \to 0$ with probability one.
\end{prop}
\subsection{The viscous transport equation}
\begin{thm}\label{thm:existence_HJB}
Let   $u_T \in C(\R^d)$ and define $u:[0,T]\times\R^d\lra\R$ by
	\begin{equation}\label{eq:formula_u}
		u(t,x) = \EE[u_T(\Phi_{T,t}(x))].
	\end{equation}
	  If $(b^\eps)_{\eps > 0}$ satisfy \eqref{eq:breg} and $u^\eps$ is the   solution of \eqref{eq:HJ} with velocity $b^\eps$, then, as $\eps \to 0$, $u^\eps$ converges locally uniformly to $u$. 
	  Moreover, if $u_T \in C^{0,1}(\R^d)$, then 
	\[
	\sup_{(t,x,y) \in [0,T] \times \R^d \times \R^d} \frac{|u(t,x) - u(t,y)|}{|x-y|} + \sup_{(r,s,x) \in [0,T] \times \R^d} \frac{|u(r,x) - u(s,x)|}{|r-s|^{1/2}(1 + |x|) } < \infty,
	\]
	and $u$ is a distributional solution of \eqref{eq:HJ}.
\end{thm}
Although \eqref{eq:formula_u} represents a distributional solution that arises uniquely through the regularization of $b$, distributional solutions are generally not unique. In \cite{ls}, it is proved that the  function $u$ given by \eqref{eq:formula_u} can be uniquely identified as the solution of \eqref{eq:HJ} in   time measurable viscosity sense.
We define, for $(t,x,p) \in [0,T] \times \R^d \times \R^d$,
\begin{equation}\label{eq:relaxed_b}
	\underline{b}(t,x,p) = \liminf_{z\to x} \hskip -2pt_{*}  b(t,z) \cdot p \quad \text{and} \quad \overline{b}(t,x,p) = \limsup_{z \to x}\hskip -2pt^{*} b(t,z) \cdot p.
\end{equation}
For fixed $(t,x) \in [0,T] \times \R^d$, $\underline{b}(t,x,\cdot)$ and $\overline{b}(t,x,\cdot)$ are Lipschitz continuous on $\R^d$, and, for fixed $(t,p) \in [0,T]$, $\underline{b}(t,\cdot,p)$ and $\overline{b}(t,\cdot,p)$ are respectively lower semicontinous (l.s.c.) and upper semicontinuous (u.s.c.).
\begin{defn}\label{def:visc}
	An u.s.c. (resp., l.s.c.) function $u$ is called a viscosity subsolution (resp. supersolution) of \eqref{eq:HJ} if   $u(\cdot,T)\le u_T$ (resp., $\ge$) and  for all $\psi\in C^{2}( \R^d)$ 
\begin{multline*}
	-\frac{d}{dt} \max_{x \in \R^d} \left\{ u(t,x)  - \psi(x) \right\}
	\\
\le \inf\left\{ \tr[ a(t,y) D^2 \psi(y)] - \underline{b}(t,y, D \psi(y)) : y \in \arg\max\{ u(t,\cdot) - \psi(\cdot) \} \right\}
\end{multline*}
\begin{multline*}
		\Big(resp.\quad -\frac{d}{dt} \min_{x \in \R^d} \left\{ u(t,x)  - \psi(x) \right\}
	\\
\ge \sup\left\{ \tr[ a(t,y) D^2 \psi(y)] - \overline{b}(t,y, D \psi(y)) : y \in \arg\min\{ u(t,\cdot) - \psi(\cdot) \} \right\} \Big)
\end{multline*}
	in   distributional sense. 
	If $u \in C([0,T] \times \R^d)$ is both a subsolution and supersolution, we say that $u$ is a viscosity solution.
\end{defn}
We have the following Comparison Principle:
\begin{thm}\label{thm:comparison}
	If $u$ and $v$ are respectively a sub and supersolution of \eqref{eq:HJ} such that
	\[
	\sup_{(t,x) \in [0,T] \times \R^d} \frac{ u(t,x)}{1 + |x|} + \sup_{(s,y) \in [0,T] \times \R^d} \frac{ - v(s,y)}{1 + |y|} < \infty,
	\]
	then $t \mapsto \sup_{x \in \R^d} \left\{ u(t,x) - v(t,x) \right\}$ is nondecreasing.
\end{thm}
As a consequence of the previous result, the ``good'' distributional solution of \eqref{eq:HJ} can be uniquely characterized.

\begin{cor}
	Assume $u_T \in C(\R^d)$ is uniformly continuous and $u_T \cdot (1 + |x|)^{-1} \in L^\infty(\R^d)$. Then \eqref{eq:formula_u} is the unique viscosity solution of \eqref{eq:HJ}.
\end{cor}
\begin{rem}
	As noted in \cite[Subsection 3.1.3]{ls}, there generally exist distributional solutions to \eqref{eq:HJ} that are not viscosity solutions to the problem. The uniqueness of distributional solutions can be ensured by requiring that $b \in C([0,T] \times \mathbb{R}^d)$ satisfies the  (OSLC) condition in \eqref{hyp:OSLC} and that $u_T \in C^{0,1}(\mathbb{R}^d)$.
\end{rem}
\subsection{The conservative equation}
For the conservative equation \eqref{eq:FP}, the tendency of the backward flow to concentrate on sets of zero Lebesgue measure suggests that, even if $f_0$ is absolutely continuous with respect to the Lebesgue measure, the solution  $f(t, \cdot)$  may develop a singular component for $t$ positive. Hence, it is natural the introduce a notion of solution in the sense of measure.
\begin{defn}\label{def:dual_sol}
	A map $f \in C([0,T], \mM)$ is called a solution of \eqref{eq:FP} if, for all $t \in [0,T]$ and $g \in C_b(\R^d)$,
\[
\int_{\R^d} g(x) f(t, dx) = \int_{\R^d} \EE[g(\Phi_{t,0}(x))] f_0(dx),
\]
where $\Phi_{t,0}$ is the stochastic flow satisfying \eqref{eq:SDE}.
\end{defn}
Solutions as in the previous definition are called duality solutions because $\EE[ g  (\Phi_{t,0}(x))]$ is the solution of \eqref{eq:HJ} with terminal value $g$ at time $t$ (cfr. \cite{CFGM}). If $f_0$ is a probability measure, then $f(t,\cdot)$ is the law of the stochastic process $\Phi_{t,0}(X_0)$, where $X_0$ is a random variable with law $f_0$.
\begin{thm}\label{thm:existence_FP}
Assume that $f_0\in \mM$. Then, there exists a unique duality solution $f$ of \eqref{eq:FP}. If, for $\eps > 0$, $f^\eps$ is the solution corresponding to $b^\eps$ as in \eqref{eq:breg}, then, as $\eps \to 0$, $f^\eps$ converges weakly in the sense of measures to $f$. If $1 \le p \le \infty$, $f_0, g_0 \in \mP_p$, and $f$ and $g$ are the corresponding duality solutions, then, for some $C > 0$ depending on $p$ and the constants in \eqref{hyp:OSLC}, $\mW_p(f(t), g(t)) \le C \mW_p(f_0,g_0)$.
\end{thm}
\begin{rem}\label{rem:first_order}
Note that, in the conservative first-order case, the results presented in this section apply to the equation
 \begin{equation}\label{eq:FP_1}
 \frac{\partial f}{\partial t} + \diver(-b(t,x) f) = 0 \quad \text{in } (0,T) \times \mathbb{R}^d, \quad \text{with } f(0,\cdot) = f_0.
 \end{equation}
 therefore, consistently with the corresponding second-order problem \eqref{eq:FP} which requires the use of the  forward  flow defined by \eqref{eq:SDE}, the velocity field, also in this first-order case is $-b$. This notation choice has been made to keep all the study consistent with \cite{ls}.
\end{rem}
\section{The numerical scheme}
In this section, we  consider semi-Lagrangian  schemes for \eqref{eq:HJ} and \eqref{eq:FP} defined on an unstructured mesh. Given $N \in \NN$ we set $h:=T/N$ and $t_k:=kh$  for $k=0, \hdots, N$. Moreover, given $\Dx>0$, we consider a triangular mesh $\mT^\Dx = (T_\tau)_{\tau \in \mK^\Dx}$ such that $\bigcup_{\tau \in \mK^\Dx} T_\tau = \R^d$,    $T_\tau \cap T_\lambda = \emptyset$ if $\tau\neq \lambda$. We assume that $\sup_\tau[\text{diam}( T_\tau)]\le \Dx$ and    $\inf_\tau [\text{height}(T_\tau)] >0$. We set $\Delta=(h,\Dx)$. \\
Given the set of the vertices of the triangulation $\mV^\Dx=\{x_i\}_{i\in \mI^\Dx}$, we consider  a   $\mathbb{Q}_1$ basis  $(\beta_{i})_{i\in \mI^\Dx}$, where   $\beta_i:\R^d \to\R$  is a polynomial of degree less than or equal to $1$ and satisfies   $\beta_i(x_j)=1$ if $i=j$ and $\beta_i(x_j)=0$, otherwise. Moreover, the support $\mbox{supp}(\beta_i)$ of $\beta_i$ is compact and  
\begin{equation}\label{eq:baricentric}
	0\leq \beta_i \leq 1  \quad \forall \; i\in \mI^\Dx, \qquad\sum_{i\in  \mI^\Dx} \beta_i(x)=1 \hspace{0.3cm} \forall \; x\in \R^d.
\end{equation}
Let $B(\mV^\Dx)$ be the space of bounded functions on $\mV^\Dx$ and, for $\phi\in B(\mV^\Dx)$, let $ \phi(i)$ be its value at $x_i$. We consider the following  linear interpolation operator 
\begin{equation*} 
I[\phi](\cdot):=\sum_{i\in  \mI^{\Dx}}\beta_i(\cdot)\, \phi(i) \qquad \mbox{for } \phi \in B(\mV^\Dx) .
\end{equation*}
Given $\phi \in C_{b}(\R^{d})$, let us define $\hat{\phi} \in B(\mV^\Dx)$ by $\hat{\phi}(i):=\phi(x_i)$ for all $i \in \mI^\Dx$. If $\phi: \R^{d} \lra \R$ is Lipschitz with constant $L$, then
\begin{equation}\label{eq:Lipschitz} 
	\sup_{x\in \R^{d}}| I[\hat{\phi}](x)-\phi(x)|\le L \Dx.
\end{equation} 
Moreover,  if $\phi\in C^2(\R^d)$ with bounded second order derivatives, then
\begin{equation}\label{eq:bound_interp_2}
	\sup_{x\in \R^{d}}| I[\hat{\phi}](x)-\phi(x)|\le C\Dx^2
\end{equation}
with $C$ independent of $\Dx$.\\
We consider  a standard mollifier   $\rho \in C^\infty(\R^d)$, $\rho \ge 0$, $\supp (\rho )\subset B(0,1)$, and $\int_{\R^d} \rho dx = 1$ and we set $\rho_h = h^{-d} \rho(\cdot/h)$ where $h$ is the time step. We define 
\begin{equation}\label{eq:integrated_coefficient}
	\begin{split}
	&b^k_h(x)=\frac{1}{h}\int_{t_k}^{t_{k+1}}b_h(s,x)ds\quad\text{where}\quad b_h(s,x) = (b(s,\cdot) * \rho_h)(x),\\
	&\sigma^k_{h,\ell}(x)=\frac{1}{h}\int_{t_k}^{t_{k+1}}\sigma_\ell(s,x)ds, \quad \ell=1,\dots,r.
		\end{split}
\end{equation} 
We approximate the stochastic flow \eqref{eq:SDE} by   a  discrete-time and countably-state space Markov chain $\{\Phi^\D_{k,m}(x_s)\; ;\; k=m, \hdots, N\}$ with state space $\mV^\Dx$  and transition probabilities
 \begin{equation}\label{eq:markov_chain}
 	\begin{split}
 &p_{ji}^{k}:= \PP\left( \Phi^\D_{k+1,m}(x_s)=x_i \; \big| \; \Phi^\D_{k,m}(x_s)=x_j\right) \\ &\qquad=\frac{1}{2r}\sum\limits_{\ell=1}^{r}\left[\beta_{i} ( \X^{k}_{\ell,+}(x_j))+\beta_{i}(\X^{k}_{\ell,-}(x_j))\right], \quad \; \forall \; i, j \in \mI^\Dx,\\
 &\PP\left( \Phi^\D_{m,m}(x_s)=x_j\right)=\delta_{s,j},
 	\end{split}
 \end{equation}
where, for $x\in\R^d$, $ k=0, \hdots, N-1$ and  $\ell=1,\hdots,r$
\begin{equation}\label{eq:discrete_traject}
	\begin{array}{rcl}
		\X^{k}_{\ell,+}(x) &:=& x - h b_h^k(x)+ \sqrt{r h }\sigma^k_{h,\ell}(x), \\[4pt]
		\X^{k}_{\ell,-}(x) &: =& x - h b_h^k(x)-\sqrt{r h }\sigma^k_{h,\ell}(x).
	\end{array}
\end{equation}
\begin{rem}
The previous approximation is motivated by the following construction. Given   the equation  \eqref{eq:SDE}, we approximate
$$
\Phi_{t+h,t}(x)=\Phi_{t ,t}(x)- \int_{t}^{t+h} b(s,\Phi_{s ,t}(x) ) d s + \int_{t}^{t+h}\sigma(s,\Phi_{s ,t}(x)) d W_s
$$	
by the $r$-dimensional random walk with $N$ time steps   
	\begin{equation}
		 \label{eq:Euler}
		 \begin{split}
	\tilde\Phi^h_{t_{k+1},t_k}(x_i)& = \tilde\Phi^h_{t_k,t_k}(x_i)- h \int_{t_k}^{t_{k+1}}b(s,\tilde\Phi^h_{t_k,t_k}(x_i))ds+ \sqrt{rh} \int_{t_k}^{t_{k+1}}\sigma(s,\tilde\Phi^h_{t_k,t_k}(x_i))Zds\\
		 &=x_i -h b^k_h(x_i) + \sqrt{rh} \sigma^k_h(x_i)Z,
		  \end{split}
	\end{equation}
where $b^k_h$, $\sigma^k_h$ defined as in \eqref{eq:integrated_coefficient} and $Z=(Z_1,\dots,Z_r)$ is a  $r$-valued random variable  satisfying   for all $\ell=1, \hdots, r$, 
\begin{equation}\label{eq:prob_walk}
\mathbb{P}( Z_{\ell}=1)= \mathbb{P}(Z_{\ell}=-1)= \frac{1}{2r} \hspace{0.5cm} \mbox{and } \hspace{0.5cm} \mathbb{P}\left( \bigcup_{1 \leq \ell_1 < \ell_2 \leq r} \{ Z_{\ell_1} \neq 0 \} \cap  \{ Z_{\ell_2} \neq 0 \} \right)=0. 
\end{equation}
If $\tilde\Phi^h_{t_{k+1},t_k}(x_i)$ is not a point of the grid, we interpolate the value by means of the basis $\beta_i$ to get the Markov chain \eqref{eq:markov_chain}. For more details, we refer to
\cite{cf,cs2}.
\end{rem}

\subsection{The numerical scheme for the viscous transport equation}
For a function $u:\{0,\dots,N\}\times B(\Ver)\lra \R$, $u_k(i)$ denotes its value at time step $k$ and grid point $x_i$. We consider the following backward explicit scheme for \eqref{eq:HJ} 
\begin{equation}\label{eq:scheme_HJ}
	\begin{cases}
		u^\D_{k}(i)=S^\D (u^\D_{k+1},i,k)    & \mbox{for all } i \in  \mI^\D,   \;  k=0, \hdots, N-1,\\[6pt]
		u^\D_{N}(i)= u_T(x_{i}), &  \mbox{for all } i \in  \mI^\D,
	\end{cases}
\end{equation}
where $S^\D : B(\Ver)\times \mI^\Dx \times \{0,\hdots, N-1\} \lra \R$ is defined as
\begin{equation}
\begin{array}{rl}
S^\D(w,i,k):=&   \frac{1}{2r} \sum_{\ell=1}^{r} \left(I [w](\X^{k}_{\ell,+}(x_i))+  I[w](\X^{k}_{\ell,-}(x_i) )\right) \\
=&\sum_{j\in  \mI^\Dx}\frac{1}{2r} \sum_{\ell=1}^{r} \left(\beta_j(\X^{k}_{\ell,+}(x_i))+  \beta_j(\X^{k}_{\ell,-}(x_i) )\right)w_j.
\end{array}
\end{equation}
and  $\X^{k}_{\ell,\pm}(x_i)$ defined in \eqref{eq:discrete_traject}. Recalling \eqref{eq:markov_chain}, the   scheme \eqref{eq:scheme_HJ} can be rewritten as
\begin{equation}\label{eq:scheme_HJ_bis}
u^\D_{k}(i)= \sum_{j\in I^\D}p^{k}_{ij}u^\D_{k+1}(j).
\end{equation}

\subsection{The numerical scheme for the conservative equation}
Now we describe the derivation of the numerical scheme for   \eqref{eq:FP}.
Given the initial measure $f_0$, to each simplex $T_\tau \in \mT^\Dx$, we associate the barycenter $B_\tau^\D$ and we define
$f_\tau^B=\int_{T_\tau} f_0(dx) $. We define  
\begin{equation}\label{eq:approx_init_meas}
	f^\Delta_{0}=\sum_{i}f^\Delta_{0}(i)\delta_{x_i}\quad \hbox{where}\quad
	f^\Delta_{0}(i)=\sum_{\tau\in \mK^\Dx}\beta_i(B_\tau^\D)f_\tau^B,
\end{equation}
i.e., the value $ f^\Delta_{0}(i)$ is calculated as the sum of the weights $ f_\tau^B $, each scaled by the barycentric coordinate $ \beta_i(B_\tau^\D) $, for all triangles in which $ x_i $ is a vertex.
We have by \eqref{eq:baricentric}
\begin{equation}\label{eq:approx_init_meas2}
	\sum_{i\in\mI^\Dx}f^\Delta_{0}(i)=\sum_{i\in\mI^\Dx}\sum_{\tau\in \mK^\Dx}\beta_i(B_\tau^\D)f_\tau^B=\sum_{\tau\in \mK^\Dx}f_\tau^B=\int_{\Rd}   f_0(dx).
\end{equation}
Given $f^\Delta_{0}(i)$, we consider the following explicit scheme  
 \begin{equation}\label{eq:scheme_FP}
	f^\Delta_{k+1}(i)= \frac{1}{2r}\sum\limits_{\ell=1}^{r}\sum\limits_{j \in \mI^\Dx} \left[\beta_{i} ( \X^{k}_{\ell,+}(x_j))+\beta_{i}(\X^{k}_{\ell,-}(x_j))\right] f^\Delta_{k}(j)  \hspace{0.4cm} \forall \; i\in \mI^\Dx, \; \; k=0,\hdots, N-1.
\end{equation}
The measure in $\mP_1(\R^d)$ defined by
\begin{equation}\label{eq:approx_measure}
	f_k^\D=\sum_{i\in\mI^\Dx}  f^\Delta_{k}(i)\delta_{x_i} \qquad k=0,\dots,N,
\end{equation}
gives  an approximation  of the   measure $f(kh,dx)$, solution to the equation \eqref{eq:FP}. To explain the scheme \eqref{eq:scheme_FP}, let us assume that, at a given step $ k $, we have computed the  measure $ f_k^\Dx $ on the discrete set $ \mV^\Dx $. At the next time step, for each vertex $ x_i $, we follow the characteristic path using the Euler scheme \eqref{eq:discrete_traject}. The mass $ f_{k}(i) $ is then distributed linearly among the vertices of the triangles $ T_\tau $ that are reached by the trajectories $ \X_{k}^{\ell,\pm}(x_i) $. The updated measure $ f_{k+1}^\Dx $ at a vertex $ x_j   $ is determined by the sum of the mass transported along the characteristic paths that reach   one of the  triangles whose the point is a vertex, with the distribution based on the barycentric coordinates of the terminal point of the characteristics. By \eqref{eq:markov_chain}, the scheme \eqref{eq:scheme_FP} can be also rewritten as
	\begin{equation}\label{eq:scheme_FP_bis}
		f^\Delta_{k+1}(i)= \sum_{j\in I^\D}p^{k}_{ji}f^\Delta_{k}(j).
	\end{equation}

\section{Convergence of the scheme for the viscous transport equation}
We aim to prove the convergence of the scheme  \eqref{eq:scheme_HJ}  to the ``good''   solution of   equation \eqref{eq:HJ}. In the following, we replace the sublinear growth conditions in \eqref{hyp:OSLC} and
\eqref{hyp:sigma} with the stronger assumptions
  \begin{align}
  	&\sup_{x \in \R^d}  |b(t,x)|\le C_0(t)\label{hyp:b_bounded},\\
  	&\sup_{x \in \R^d}  |\sigma(t,x)| \le C_2(t) \label{hyp:sigma_bounded},
  \end{align}
 with $C_0, C_2\in L^\infty(0,T)$, i.e. $b$, $\sigma\in L^\infty([0,T], L^\infty(\R^d))$.  Note that  a similar assumption concerning \( b \) is also made in \cite{dlv, dl}.  
\par
To prove the convergence of the scheme, we use an intermediate, semi-discrete in time scheme. On one hand, for the solution of this scheme, it is possible to demonstrate certain properties similar to those satisfied by the continuous problem, and thus prove its convergence through a stability argument; in this part, the property \eqref{hyp:OSLC} plays a key role. On the other hand, the regularity of the solution to the semi-discrete scheme allows us to estimate, in the $ L^\infty $-norm, its distance from the solution of \eqref{eq:scheme_HJ}.
\par
We consider the semi-discrete scheme
\begin{equation}\label{eq:scheme_HJ_semi}
	\begin{cases}
		u^h_{k}(x)=S^h (u^h_{k+1}(\cdot),x,k)    & \mbox{for all } x\in\R^d,   \;  k=0, \hdots, N-1,\\[4pt]
		u^h_{N}(x)= u_T(x), &  \mbox{for all } x\in\R^d,
	\end{cases}
\end{equation}
where $S^h : L^\infty(\R^d)\times \R^d \times \{0,\hdots, N-1\} \lra \R$ is defined as
\begin{equation}\label{eq:op_discreto}
	\begin{array}{rl}
		S^h(w,x,k)&:= \frac{1}{2r} \sum_{\ell=1}^{r} \left(w(\X^{k}_{\ell,+}(x))+  w(\X^{k}_{\ell,-}(x) )\right)\\
		 &=\frac{1}{2r} \sum_{\ell=1}^{r} \big[w(x  - h b_h^k(x )+ \sqrt{r h }\sigma^k_{h,\ell}(x ) ) 
		+w(x  - h b_h^k(x )-\sqrt{r h }\sigma^k_{h,\ell}(x ) )\big].
	\end{array}
\end{equation}
Define a Markov chain by
\begin{equation} \label{eq:Euler_chain}
	\begin{cases}
		\Phi^h_{m+1,k}(x)= \Phi^h_{m,k}(x)- h b^m_h(\Phi^h_{m,k}(x)) + \sqrt{rh} \sigma^m_h(\Phi^h_{m,k}(x))Z^m  & m=k+1,\dots, N-1\\
		 \Phi^h_{k,k}(x)=x,
	\end{cases}
\end{equation}
where $Z^m$ is a sequence of $r$-valued random variables  independent of $\Phi^h_{m,k}(x)$  and  satisfying \eqref{eq:prob_walk}, while $b_h^m$, $\sigma_h^m$ are defined as in \eqref{eq:integrated_coefficient} with $m$ in place of $k$. Note that $\Phi^h_{k+1,k}(x)=\tilde \Phi^h_{t_{k+1},t_k}(x)$
where $\tilde \Phi^h$ is defined in \eqref{eq:Euler}.

\begin{prop}\label{prop:exist_semidiscrete_HJB}
Assume that $u_T\in C_b(\R^d)$. For $h>0$, 	let $u^h=\{u^h_k\}_{k=0}^N$  be the solution of \eqref{eq:scheme_HJ_semi}. Then
\begin{equation}\label{eq:formula_rapp_discr}
	u^h_k(x)=\EE[u_T(\Phi^h_{N,k}(x))],\quad x\in\R^d,\ k=0,\dots,N 
\end{equation}
and  $u^h_k(x)$ is bounded and continuous in $\R^d$ for any $k=0,\dots,N$. Moreover, if $u_T$ is Lipschitz continuous, then $u^h$ is also Lipschitz continuous in $x$ with a constant $L$ independent of $h$.
\end{prop}
\begin{proof}
	We first prove  \eqref{eq:formula_rapp_discr}. The identity is obvious for $k=N$ since $\Phi^h_{N,N}(x)=x$. For $k=N-1$, we have
	\begin{align*}
		&\EE[u_T(\Phi^h_{N,N-1}(x))]\\
		&=\EE[u_T(\Phi^h_{N-1,N-1}(x)- h b^{N-1}_h(\Phi^h_{N-1,N-1}(x)) + \sqrt{rh} \sigma^{N-1}_h(\Phi^h_{N-1,N-1}(x))Z^{N-1})]\\
	   &	=\EE[u_T(x- h b^{N-1}_h(x) + \sqrt{rh} \sigma^{N-1}_h(x)Z^{N-1})]=\frac{1}{2r} \sum_{\ell=1}^{r} \left(u_T(\X^{k}_{\ell,+}(x))+  u_T(\X^{k}_{\ell,-}(x) )\right)\\
	   &=S^h(u_N,x,N-1)=u^h_{N-1}(x) 
	\end{align*}
and therefore  \eqref{eq:formula_rapp_discr} for $k=N-1$. For $k=N-2$, observing that $\Phi^h_{N,N-2}=\Phi^h_{N,N-1}\circ \Phi^h_{N-1,N-2}$ and since \eqref{eq:formula_rapp_discr} holds for $k=N-1$, we have 
\begin{align*}
&\EE[u_T(\Phi^h_{N,N-2}(x))]=\EE[u_T(\Phi^h_{N,N-1}(\Phi^h_{N-1,N-2}(x)))]\\
&=\EE[u^h_{N-1}( \Phi^h_{N-1,N-2}(x))]=\EE[u^h_{N-1}( x- h b^{N-2}_h(x) + \sqrt{rh} \sigma^{N-2}_h(x)Z^{N-2})]\\
&=\frac{1}{2r} \sum_{\ell=1}^{r} \big[u^h_{N-1}(x  - h b_h^{N-2}(x )+ \sqrt{r h }\sigma^{N-2}_{h,\ell}(x ) ) 
+u^h_{N-1}(x  - h b_h^{N-2}(x )-\sqrt{r h }\sigma^{N-2}_{h,\ell}(x ) )\big]\\
&=S^h(u_{N-1},x,N-2)= u^h_{N-2}(x),
\end{align*}
which gives \eqref{eq:formula_rapp_discr} for $k=N-2$. Iterating,   we get \eqref{eq:formula_rapp_discr} for $k=0,\dots, N-1$.\\
 By the assumption on $u_T$ and \eqref{eq:formula_rapp_discr}, it follows immediately that $\sup_{k=0,\dots,N}\|u^h_k\|_{L^\infty}\le \|u_T\|_{L^\infty}$. To prove the  continuity of $u^h_k$, $k=0,\dots,N$, we first give a stability estimate for the discrete trajectories \eqref{eq:Euler_chain}. Given $x,y\in\R^d$, we denote by $\Phi^h_{m,k}$, $\Psi^h_{m,k}$ the discrete trajectories such that $\Phi^h_{k,k}=x$, $\Psi^h_{k,k}=y$. We have
\begin{equation}\label{eq:dp_discrete1}
	\begin{split}
	&\EE|\Phi^h_{m+1,k}-\Psi^h_{m+1,k}|^2	=\EE|\Phi^h_{m,k}-\Psi^h_{m,k}|^2+h^2   \EE|b_h^m(\Phi^h_{m,k})-b_h^m(\Psi^h_{m,k})|^2\\
	&+hr\EE|(\sigma_h^m(\Phi^h_{m,k})-\sigma_h^m(\Psi^h_{m,k}))Z^m|^2-2h \EE\big[ (\Phi^h_{m,k}-\Psi^h_{m,k})(b_h^m(\Phi^h_{m,k})-b_h^m(\Psi^h_{m,k}))\big]\\
	&+2\sqrt{hr}\,\EE\big[(\Phi^h_{m,k}-\Psi^h_{m,k})(\sigma_h^m(\Phi^h_{m,k})-\sigma_h^m(\Psi^h_{m,k}))Z^m\big]\\
	&+2\sqrt{h^3r}\,\EE\big[(b_h^m(\Phi^h_{m,k})-b_h^m(\Psi^h_{m,k}))(\sigma_h^m(\Phi^h_{m,k})-\sigma_h^m(\Psi^h_{m,k}))Z^m\big].
	\end{split}
\end{equation}
By \eqref{hyp:b_bounded}, we estimate 
\begin{align*}
h^2\EE&\left|b_h^m(\Phi^h_{m,k})-b_h^m(\Psi^h_{m,k})\right|^2=h^2\EE\left|\frac{1}{h}\int_{t_m}^{t_{m+1}}\int_{\R^d}b(s,y)(\rho_h(\Phi^h_{m,k}-y)-\rho_h(\Psi^h_{m,k}-y))dyds\right|^2\\
&\le h^2\EE\left|\frac{1}{h}\int_{t_m}^{t_{m+1}}|b(s,y)|\left(\int_{B(0, h)}\|D\rho_h\|_{\infty}dy\right)|\Phi^h_{m,k}- \Psi^h_{m,k}|ds\right|^2\\
&\le h^2\EE\left|\frac{1}{h}\int_{t_m}^{t_{m+1}} C_0(s)|\Phi^h_{m,k}- \Psi^h_{m,k}|ds\right|^2\le h\EE|\Phi^h_{m,k}-\Psi^h_{m,k}|^2 \left( \frac{1}{h}\int_{t_m}^{t_{m+1}}C_0(s)ds\right)^2.
\end{align*}
By \eqref{hyp:sigma} and \eqref{hyp:sigma_bounded}, since $Z^m$ independent of $\Phi^h_{m,k}-\Psi^h_{m,k}$, we estimate
\begin{align*}
hr\EE&|(\sigma^m(\Phi^h_{m,k})-\sigma^m(\Psi^h_{m,k}))Z^m|^2=hr
\EE\left|\frac{1}{h}\int_{t_m}^{t_{m+1}}(\sigma(s,\Phi^h_{m,k})-\sigma(s,\Psi^h_{m,k}))Z^mds\right|^2\\
&\le hr\EE|\Phi^h_{m,k}-\Psi^h_{m,k}|^2\left|\frac{1}{h}\int_{t_m}^{t_{m+1}}C_2(s)ds\right|^2\EE[ |Z^m| ^2].
\end{align*}
By \eqref{hyp:OSLC}, it follows that
\begin{align*}
	- 2h \EE(\Phi^h_{m,k}-\Psi^h_{m,k})(b_h^m(\Phi^h_{m,k})-b_h^m(\Psi^h_{m,k}))\le 2h \EE|\Phi^h_{m,k}-\Psi^h_{m,k}|^2 \left(\frac{1}{h}\int_{t_m}^{t_{m+1}}C_1(s)ds\right).
\end{align*}
Since $Z^m$ is independent of $\Phi^h_{m,k}-\Psi^h_{m,k}$ and $\EE[Z^m]=0$, the last two terms on the right-hand side of the identity \eqref{eq:dp_discrete1} vanish. Hence, replacing the previous estimates in \eqref{eq:dp_discrete1}, we finally get the stability estimate
\[
\EE|\Phi^h_{m+1,k}-\Psi^h_{m+1,k}|^2\le (1+Ch)\EE|\Phi^h_{m,k}-\Psi^h_{m,k}|^2
\]
and, iterating on $m$,
\begin{equation}\label{eq:dp_discrete2}
	\EE|\Phi^h_{m+1,k}-\Psi^h_{m+1,k}|^2\le (1+Ch)^{m+1-k}|x-y|^2, \qquad m=k-1,\dots, N-1
\end{equation}
for $C$ depending only on the constants in assumptions \eqref{hyp:OSLC}, \eqref{hyp:sigma},  \eqref{hyp:b_bounded} and \eqref{hyp:sigma_bounded}.
\par
We also estimate the dependence of $\Phi^h_{m,k}$ with respect to $m$. Given $m\in \{k,\dots,N\}$, we have
\begin{align*}
&\EE|\Phi^h_{m+1,k}(x)-\Phi^h_{m,k}(x)|^2=h^2\EE|b_h^m(\Phi^h_{m,k})|^2+hr\EE|\sigma^m(\Phi^h_{m,k})Z^m|^2\\
&+2h\sqrt{hr}\EE[b_h^m(\Phi^h_{m,k})\sigma^m(\Phi^h_{m,k})Z^m]=h^2\EE\left|\frac{1}{h}\int_{t_m}^{t_{m+1}}\int_{\R^d}b(s,y)\rho_h(\Phi^h_{m,k}-y)dyds\right|^2\\
&+hr\EE\left|\frac{1}{h}\int_{t_m}^{t_{m+1}}\sigma(s,\Phi^h_{m,k})Z^mds\right|^2\le h^2 \left|\frac{1}{h}\int_{t_m}^{t_{m+1}} C_0(s)ds\right|^2+hr \left[\frac{1}{h}\int_{t_m}^{t_{m+1}}|C_2(s)|^2ds \right]
\le Ch
\end{align*}
with $C$ independent of $h$. In similar way, we estimate
\begin{equation}\label{eq:dp_discrete3}
	\EE|\Phi^h_{m+\bar m,k}(x)-\Phi^h_{m,k}(x)|^2\le C\bar m h,\qquad m=k,\dots, N,\ \ 0\le \bar m\le N-m.
\end{equation}
The estimate \eqref{eq:dp_discrete2} and the representation formula \eqref{eq:formula_rapp_discr} give immediately the continuity of $u^h_k$ for any $k=0,\dots,N$. Moreover, if $u_T$ is Lipschitz continuous with constant $L_T$, then, given $x,y\in\R^d$, we have by \eqref{eq:dp_discrete2}
\begin{equation}\label{eq:stima_lip}
	\begin{split}
	|u^h_k(x)-u^h_k(y)|&\le \EE|u_T(\Phi^h_{N,k}(x))-u_T(\Phi^h_{N,k}(y))|\le
	L_T[\EE| \Phi^h_{N,k}(x) - \Phi^h_{N,k}(y) |^2]^\half\\
	&\le L_Te^{C(N-k)h}|x-y|\le L |x-y|
	\end{split}
\end{equation}
with $L$ independent of $h$. Hence the uniform Lipschitz property for $u^h_k$, $k=0,\dots,N$ follows.
\end{proof}
We now estimate the distance between the solutions of \eqref{eq:scheme_HJ} and \eqref{eq:scheme_HJ_semi}.
\begin{prop}\label{prop:semi-to-fully}
Assume that $u_T$ is   Lipschitz continuous. Let  $\bar u^{\D}$  be the linear interpolation  of the solution of \eqref{eq:scheme_HJ} on the triangulation $\mT^\Dx$, i.e. $\bar u^{\D}_k(x)=I[u^\D_k](x)$,  and  $u^h$   the solution of \eqref{eq:scheme_HJ_semi}. Then
\begin{equation}\label{eq:est_discr_fullydiscr}
\|u^h_{k}-\bar u^{\D}_k\|_{L^\infty{(\R^d)}}\le L(N-k+1)\Dx,\qquad k=0,\dots,N,
\end{equation}
where $L$ is as in Prop. \ref{prop:exist_semidiscrete_HJB}.
\end{prop}
\begin{proof}
At step $N$, the estimate follows from \eqref{eq:Lipschitz}. At step $N-1$, we have for $x\in\R^d$
	\begin{equation}\label{eq:est_1}
		\begin{split}
			&|u^h_{N-1}(x)-\bar u^{\D}_{N-1}(x)|=|u^h_{N-1}(x)-\sum_{i\in\mI^\Dx}\beta_i(x)\bar u^\D_{N-1}(x_i)|\\
			& \le \sum_{i\in\mI^\Dx}\beta_i(x)|u^h_{N-1}(x)-u^h_{N-1}(x_i)| 
			+\sum_{i\in\mI^\Dx}\beta_i(x)|u^h_{N-1}(x_i)-\bar u^\D_{N-1}(x_i)|.
			\end{split}
	\end{equation}
Since $\beta_i(x)\neq 0$ if and only if $x_i$ is a vertex of the triangle containing $x$, by \eqref{eq:stima_lip} we have
\begin{equation}\label{eq:est_2}
	|u^h_{N-1}(x)-u^h_{N-1}(x_i)|\le L\Dx.
\end{equation}
Moreover, since $\bar u^{\D}_{N-1}(x_i)=u^{\D}_{N-1}(x_i)$,  we have
\begin{align*}
&u^h_{N-1}(x_i)=\frac{1}{2r} \sum_{\ell=1}^{r} \left(u_T(\X^{k}_{\ell,+}(x_i))+  u_T(\X^{k}_{\ell,-}(x_i) )\right),	\\
&\bar u^\D_{N-1}(x_i)=\frac{1}{2r} \sum_{\ell=1}^{r} \left(I [\hat u_T](\X^{k}_{\ell,+}(x_i))+  I[\hat u_T](\X^{k}_{\ell,-}(x_i) )\right).
\end{align*}
Subtracting the previous equation and using \eqref{eq:Lipschitz}, we get
\begin{equation}\label{eq:est_3}
|u^h_{N-1}(x_i) -u^\D_{N-1}(x_i)|\le L_T\Dx.
\end{equation}
Plugging \eqref{eq:est_2} and \eqref{eq:est_3} into \eqref{eq:est_1}, we get \eqref{eq:est_discr_fullydiscr} for $k=N-1$. Following a similar argument for $k=N-2, N-3,\ldots$ and exploiting the estimate obtained at the previous step to bound $u^h_{k}(x_i) -\bar u^\D_{k}(x_i)$, we obtain \eqref{eq:est_discr_fullydiscr} for any $k=0,\dots,N$.
\end{proof}
We now discuss the convergence of $\bar u^\D$ to $u$. We first  extend  $u^h$ and $\bar u^\D$ to $[0,T]\times \R^d$ by setting, for $\phi_k=u^h_k,\, \bar u^\D_k$,
\begin{align*}
\phi(t,x)=\frac{(k+1)h-t}{h}	\phi_k(x) +\frac{t-kh}{h}\phi_{k+1}(x),\quad t\in [kh, (k+1)h), \ \ k=0,\dots,N-1.
\end{align*}
We use a definition  of $L^1$-viscosity solution equivalent to Def. \ref{def:visc} which seems more suitable for stability properties (see \cite{lp} for the equivalence of the two definitions and \cite{b,m} for related stability properties). We set
\begin{equation}
	H(t,x,p,X)=\tr[ a(t,x) X] - b(t,x) p
\end{equation}
and
\begin{align*}
	\overline H(t,x,p,X)=\tr[ a(t,x) X] - \underline{b}(t,x, p),\\
	\underline H(t,x,p,X)=\tr[ a(t,x)X] - \overline{b}(t,x, p),
\end{align*}
where $\underline{b}$, $\overline{b}$ are defined in \eqref{eq:relaxed_b}.
\begin{defn}\label{def:L1_visco}
	An u.s.c. (respectively, l.s.c.) function  $u:\R^d\times [0,T]\lra \R$ is called a subsolution (respectively, supersolution)  of \eqref{eq:HJ} if $u(\cdot,T)\le u_T(x)$ (resp., $\ge$)
	and  
	\begin{itemize}
		\item[$(i)$] for any $\phi \in C^{1,2}((0,T)\times \R^d )$ and any $\gamma\in L^1(0,T)$ such that the function $u(t,x)-\phi(t,x)-\int_t^T\gamma(s)ds$ has a local maximum (resp., minimum) at $(t_0,x_0)\in\R^d\times (0,T)$;
		\item[$(ii)$] for any continuous function $G:(0,T)\times\R^d\times  \R^d\times S^d\lra \R$ such that
		\begin{align*}
		&	\overline H(t,x,p,X)-\gamma(t)\le G(t,x,p,X)\\
		&	\Big(\text{resp.}\quad \underline H(t,x,p,X)-\gamma(t)\ge G(t,x,p,X)\Big)
		\end{align*}
		for all $(x,p,X)$ in a neighborhood of $(x_0, D\phi(t_0,x_0), D^2\phi(t_0,x_0))$ and almost all $t$ in a neighborhood of $t_0$, we have
	\begin{align*}
	-\frac{\partial\phi}{\partial t}(t_0,x_0)\le G(t_0,x_0, D\phi(t_0,x_0), D^2\phi(t_0,x_0))\\
		\left(\text{resp.}\quad -\frac{\partial\phi}{\partial t}(t_0,x_0)\ge G(t_0,x_0, D\phi(t_0,x_0), D^2\phi(t_0,x_0))\right).
	\end{align*}	
	\end{itemize}
	If $u \in C([0,T] \times \R^d)$ is both a sub and supersolution, we say $u$ is a solution of \eqref{eq:HJ}.
\end{defn}


\begin{thm}\label{thm:fully_to_continuous_HJ}
	Assume that $u_T$ is   Lipschitz continuous. Then, 	for $|\D|=|(h,\Dx)|\to 0$ with $\Dx/h\to0$, the  function $\bar u^\D$,  given by the linear interpolation of the solution of \eqref{eq:scheme_HJ} on $[0,T]\times\R^d$,	converges to the solution $u$ of \eqref{eq:HJ}, locally uniformly in $[0,T]\times\R^d$.
\end{thm}
\begin{proof}
	By \eqref{eq:est_discr_fullydiscr}, we have
\begin{equation*}
		\sup_{t\in[0,T]}\|u^h(t)-u^\D(t)\|_{L^\infty(\R^d)}\le LT\frac{\Dx}{h}
\end{equation*}
with $L$ independent of $h$ and $\Dx$. Hence, it is sufficient to show that $u^h$ converges,  locally uniformly in $[0,T]\times\R^d$, to  $u$ and, for this, we use the    stability argument   in \cite{bs} adapted to Def. \ref{def:L1_visco}  (see  \cite{m} for a related result). We set 	
	\begin{equation*}
		\underline{u}(t,x):=\liminf_{h\to 0} \hskip -2pt_{*}    u^h(s,y),  \quad \quad \overline{u}(t,x):=\limsup_{h\to 0}\hskip -2pt^{*}      u^h(s,y) . 
	\end{equation*} 
	We verify that $\overline{u}$ is a subsolution of \eqref{eq:HJ}. Notice that, due to \eqref{eq:dp_discrete2} and \eqref{eq:dp_discrete3}, we have that $u^h$ is   uniformly   continuous   in $x$ and $t$. Hence, recalling that $u^h(T,x)=u_T(x)$,  we have $\overline{u}(T,x)=u_T(x)$.
\par
	Let $\phi\in C^{2}((0,T)\times \R^d)$ and $\gamma\in L^1(0,T)$ be such that
	$\bar u(t,x)-\phi(t,x)-\int_0^t\gamma(s)ds$ has a local maximum point at
	 $(t_0,x_0)\in (0,T)\times  \R^d $ and let $G$ be a continuous function such that
	 \begin{equation}\label{eq:G}
	 	\overline H(t,x,p,X)-\gamma(t)\le G(t,x,p,X)
	 \end{equation}
	 in a neighborhood of $(x_0, D\phi(t_0,x_0), D^2\phi(t_0,x_0))$ and for almost all $t$ in a neighborhood of $t_0$. We have to prove that
	 \begin{equation}\label{eq:subsol}
	 	\phi_t(t_0,x_0)\le G(t_0,x_0, D\phi(t_0,x_0), D^2\phi(t_0,x_0)).
	 \end{equation}
It is not restrictive to assume that $(t_0,x_0)$ is a strict global maximum point and that, additionally,   $\sup_{t\in (0,T)}\|\phi(\cdot,t)\|_{C^2(\R^d)}$ is finite. We set
\begin{equation}\label{eq:ham_discret}	 
H^h(w,x,k)=\frac{S^h(w,x,k)-w(x)}{h},
\end{equation}
where $S^h$ is defined as in \eqref{eq:op_discreto}. We have  for $\psi(\cdot)=\phi(t_0,\cdot)$
\begin{equation}\label{eq:stima_oper}
\lim_{h\to 0} \sum_{k=[t/h]}^{N-1}h H^h(\psi,x_0,k)=\int_{t}^{T}H(s,x_0,D\psi(x_0), D^2\psi(x_0))ds,\quad \forall t\in [0,T],
\end{equation}
locally uniformly in $x$ (see Lemma \ref{lem:approx} at the end of the proof). For $t\in [0,T]$, we define
\[
\mB^h(t,x)=u^h(t,x)-\phi(t,x)-\int_t^T\gamma(s)ds-\cE_h(t),
\]
where
\[
\cE_h(t)=\sum_{k=[t/h]}^{N-1}h H^h(\phi(t_0,\cdot),x_0,k)-\int_{t}^{T}H(s,x_0,D\phi(t_0,x_0), D^2\phi(t_0,x_0))ds.
\] 
By \eqref{eq:stima_oper}, since $\limsup_{h\to 0}^*\mB^h(t,x)=u^*(t,x)-\phi(t,x)-\int_t^T\gamma(s)ds$,   there exists a sequence $(t^h,x^h)$, with $t^h=k^h h$, of global maximum points for $\mB^h-\phi$ converging to $(t_0,x_0)$ for $h\to 0^+$. We set $C^h=\mB^h(t^h,x^h)$ and we observe that 
\begin{align*}
&	u^h(t,x)\le \phi(t,x)+\int_t^T\gamma(s)ds+\cE_h(t)+C^h,\qquad \forall (t,x)\in\R^d\times (0,T),\\
&	u^h(t^h,x^h)=\phi(t^h,x^h)+\int_{t^h}^T\gamma(s)ds+\cE_h(t^h)+C^h.
\end{align*}
Hence, by \eqref{eq:scheme_HJ_semi}
\begin{align*}
\phi(t^h,x^h)&+\int_{t^h}^T\gamma(s)ds+\cE_h(t^h)+C^h=u^h(t^h,x^h)=u^h_{k^h}(x^h)= S^h (u^h_{k^h+1}(\cdot),x^h,k^h) \\
&\le S^h (\phi( t^h+h, \cdot),x^h,k^h)+\int_{t^h+h}^T\gamma(s)ds+\cE_h(t^h+h)+C^h
\end{align*}
and, recalling \eqref{eq:ham_discret}, we have
\begin{align*}
-\frac{\phi(t^h+h, x^h)-\phi(t^h,x^h )}{h}&	\le H^h(\phi(t^h+h, \cdot),x^h,k^h)
-\frac{1}{h}\int_{t^h}^{t^h+h}\gamma(s)ds\\
&+\frac{1}{h}\left(\cE_h(t^h+h)-\cE_h( t^h)\right).
\end{align*}
By the definition of $\cE_h$, we get
\begin{align*}
	&-\frac{\phi(t^h+h, x^h )-\phi(t^h,x^h )}{h}	\le H^h(\phi(t^h+h,\cdot),x^h,k^h) -\frac{1}{h}\int_{t^h}^{t^h+h}\gamma(s)ds\\
	&+H^h(\phi(t_0,\cdot), x_0, k^h)+\frac{1}{h}\int_{t^h}^{t^h+h}H(x_0,s,D\phi(t_0,x_0),D^2\phi(t_0,x_0))ds\\
	&= H^h(\phi(\cdot, t^h+h),x^h,k^h)-H^h(\phi(t_0,\cdot), x_0, k^h)\\
	&+\frac{1}{h}\int_{t^h}^{t^h+h}\left(H(x_0,s,D\phi(t_0,x_0),D^2\phi(t_0,x_0))-\gamma(s)\right)ds.
\end{align*}
Passing to the limit for $h\to 0$, recalling \eqref{eq:G} and since
\[
H^h(\phi(t^h+h,\cdot),x^h,k^h)-H^h(\phi(t_0,\cdot), x_0, k^h)\lra 0\quad\text{for $h\to 0$},
\]
we get \eqref{eq:subsol}. By a similar argument, it is possible to show that $\underline u$ is a supersolution of \eqref{eq:HJ}. Hence by Theorem \ref{thm:comparison}, we get  $u=\overline u=\underline u$ and the local uniform convergence of $u^h$ to $u$.
\end{proof}
\begin{lemma}\label{lem:approx}
	Let $\psi\in C^2(\R^d)$. Then
	\begin{equation}\label{eq:conv_lemma}
		\lim_{h\to 0} \sum_{k=[t/h]}^{N-1}h H^h(\psi,x,k)=\int_{t}^{T}H(s,x,D\psi(x), D^2\psi(x))ds,\qquad \forall t\in [0,T),
	\end{equation}
	locally uniformly in $x\in\R^d$.  
\end{lemma}
\begin{proof}
 	We first observe that
	\begin{align*}
			H^h(\psi, x,k)&=\frac{1}{2r} \sum_{\ell=1}^{r}\frac{1}{h} \left[\psi(x  - h b_h^k(x )+ \sqrt{r h }\sigma^k_{h,\ell}(x ) ) 
			+\psi(x  - h b_h^k(x )-\sqrt{r h }\sigma^k_{h,\ell}(x ) )-2\psi(x)\right]\\		
	&=\tr\left[ \frac{1}{2} \sigma^h_k (\sigma^h_k)^T D^2 \psi(x)\right] - b^h_k(x) D \psi(x)+o(h)
	\end{align*}
where  $ o(h)\lra 0$ for $h\to 0$.
Hence 
\begin{align*}
\sum_{k=[t/h]}^{N-1}h H^h(\psi,x,k)&=\sum_{k=[t/h]}^{N-1}h\Bigg[\frac{1}{2}\tr\left(\frac{1}{h}\int_{kh}^{(k+1)h}\sigma(s,x)ds \frac{1}{h}\int_{kh}^{(k+1)h}\sigma^T(s,x)dsD^2\psi(x)\right)\\
&-\left(\frac{1}{h}\int_{kh}^{(k+1)h}\int_{\R^d}\rho_h(x-y)b(s,y)dyds\right)D\psi(x)\Bigg]+o(h)\\
&=\int_{t}^{T}H(x,s,D\psi(x), D^2\psi(x))ds-\int_{ [t/h]h}^tH(x,s,D\psi(x), D^2\psi(x))ds\\
&+\int_{ [t/h]h}^T\int_{\R^d}\rho_h(x-y)(b(s,y)-b(s,x))dyds\\
&+\sum_{k=[t/h]}^{N-1}h\frac{1}{2}\tr\Bigg(\frac{1}{h}\int_{kh}^{(k+1)h}\sigma(s,x)ds \frac{1}{h}\int_{kh}^{(k+1)h}\sigma^T(s,x)ds\\
&-\frac{1}{h}\int_{kh}^{(k+1)h}\sigma(s,x)\sigma^T(s,x)ds\Bigg)D^2\psi(x) +o(h)
\end{align*}
and therefore \eqref{eq:conv_lemma} follows.
\end{proof}
\section{Convergence of the scheme for the conservative equation}
The aim of this section is to show that the scheme \eqref{eq:scheme_FP} converges to the unique duality solution, in the sense of Definition \ref{def:dual_sol}, of the conservative equation \eqref{eq:FP}.
In this section, besides \eqref{hyp:OSLC}, \eqref{hyp:b_bounded}, \eqref{hyp:sigma_bounded}, we assume that
\begin{equation}\label{hyp:m_0}
	f_0\in \mP_2(\R^d) \quad \text{and} \quad \supp[ f_0]\subset B(0,R)\quad\text{for some $R>0$}
\end{equation}
where $\supp[ f_0]$ denotes the support of the measure $f_0$.
\begin{lemma}
\label{convf0}
Let  $f^\D_0$ be defined as in \eqref{eq:approx_init_meas}. Then $f^\D_0\in \mP_2(\R^d)$ and 
$\mW_1(f^\Delta_0, f_0)\lra 0$   as $\Delta\to 0$.
\end{lemma}
\begin{proof}
Recall that, by \eqref{eq:approx_init_meas},
$f^\Delta_{0}=\sum_{i}f^\Delta_{0}(i)\delta_{x_i}$ where 
$f^\Delta_{0}(i)=\sum_{\tau\in \mK^\Dx}\beta_i(B_\tau^\D)f_\tau^B$ and $f_\tau^B=\int_{T_\tau} f_0(dx)$. By \eqref{eq:approx_measure}, we have $\int_{\R^d}f^\D_0(dx)=\int_{\R^d}f_0(dx)=1$. Moreover
\begin{align*}
&\int_{\R^d}|x|^2f^\D_0(dx)=	\sum_{i\in \mI^\Dx}|x_i|^2f^\Delta_{0}(i)=\sum_{i\in \mI^\Dx}|x_i|^2\sum_{\tau\in \mK^\Dx}\beta_i(B_\tau^\D) \int_{T_\tau} f_0(dx)\\
&=\sum_{\tau\in \mK^\Dx}\int_{T_\tau}\sum_{i\in \mI^\Dx}\beta_i(B_\tau^\D) |(x_i-x)+x|^2f_0(dx).
\end{align*}
Since $|x_i-x|\le \Dx$ for $x,x_i\in T_\tau$, $\beta_i(B_\tau^\D)= 0$   if $x_i\not \in T_\tau$ and $\sum_{i\in \mI^\Dx}\beta_i(B_\tau^\D)=1$, we get
\begin{align*}
	\int_{\R^d}|x|^2f^\D_0(dx)&\le \Dx^2\sum_{\tau\in \mK^\Dx}\int_{T_\tau}\sum_{i\in \mI^\Dx}\beta_i(B_\tau^\D)f_0(dx)+\sum_{\tau\in \mK^\Dx}\int_{T_\tau}|x|^2\sum_{i\in \mI^\Dx}\beta_i(B_\tau^\D)f_0(dx)\\
	&\le\Dx^2+\int_{\R^d}|x|^2f_0(dx).
\end{align*}
We claim that $\mW_1(f^\D_0,f^\D)\lra 0$ for $\D\to 0$. For $\phi$ 1-Lipschitz, we have
\begin{align*}
\int_{\R^d}\phi(x)& (f^\Delta_0(dx)-f_0(dx))=\int_{\R^d}\phi(x)\left(\sum_{i\in \mI^\Dx}f^\Delta_{0}(i)\delta_{x_i}-f_0(dx)\right)\\
&	= \sum_{i\in \mI^\D}\beta_i(B_\tau^\D)f_\tau^B\phi(x_i)-\sum_{\tau\in \mK^\D}\int_{T_\tau}\phi(x) f_0(dx) \\
& 	=	\sum_{\tau\in \mK^\D}\left[\int_{T_\tau}\left(\sum_{i\in \mI^\D}\beta_i(B_\tau^\D)\phi(x_i)-\phi(x)\right) f_0(dx)\right].
 \\&	\le \sup_{\tau}\{ \hbox{diam}T_\tau\}\sum_{\tau\in \mK^\D}\int_{T_\tau}  f_0(dx)\le
   \Delta x \int_{\R^d}f_0(dx) 
\end{align*}
and therefore the claim follows.
\end{proof}

\begin{lemma}
Let $u^\D$ be the solution of \eqref{eq:scheme_HJ}. Then
\begin{equation}\label{eq:form_rapp_HJ}
	u^\D_{k}(i)=\EE[u_T(\Phi^\D_{N,k}(x_i))],\qquad i\in\mI^\D, \ k=0,\ldots,N,
\end{equation}	
where $\Phi^\D_{N,k}(x_i)$ is defined as in \eqref{eq:markov_chain}.
\end{lemma}
\begin{proof}
	We prove that $u^\D_{k}(i)$ defined as in \eqref{eq:form_rapp_HJ} gives a solution of \eqref{eq:scheme_HJ_bis}.	For $k=N$, since $\PP(\Phi^\D_{N,N}(x_i)=x_j)=\delta_{i,j}$, we get $u^\D_{N}(i)=u_T(x_i)$. For $k=N-1$, by \eqref{eq:markov_chain} and \eqref{eq:form_rapp_HJ} we have
	\begin{align*}
	&	\EE[u_T(\Phi^\D_{N,N-1}(x_i))]=
	\sum_{j\in \mI^\D}\PP(\Phi^\D_{N,N-1}(x_i)=x_j)u_T(x_j)\\
	&=\sum_{j\in \mI^\D}\sum_{k\in \mI^\D}\PP(\Phi^\D_{N,N-1}(x_i)=x_j|\Phi^\D_{N-1,N-1}(x_i)=x_k)\PP(\Phi^\D_{N-1,N-1}(x_i)=x_k)u_T(x_j)\\
	&=\sum_{j\in\mI^\D}\PP(\Phi^\D_{N,N-1}(x_i)=x_j|\Phi^\D_{N-1,N-1}(x_i)=x_i)u_T(x_j)=\sum_{j\in\mI^\D} p^{(N-1)}_{ij}u^\D_{N}(j)=u^\D_{N-1}(i)
	\end{align*}
and hence we get \eqref{eq:form_rapp_HJ} for $k=N-1$. Following a similar argument for $k=N-2, N-3,\ldots$, we obtain that the solution of \eqref{eq:scheme_HJ_bis} is given by the representation formula \eqref{eq:form_rapp_HJ}.
\end{proof}
\begin{lemma}
	Let $f^\D_0$ be given  by \eqref{eq:approx_init_meas} and $f^\D$     by  \eqref{eq:approx_measure}. Then, for $g\in C_b(\R^d)$, 
	\begin{equation}\label{eq:duality_discrete}
		\int_{\R^d} g(x)f^\D_k(dx)=\int_{\R^d}\EE[g(\Phi^\D_{k,0}(x))]f^\D_0(dx).
	\end{equation}	
\end{lemma}
\begin{proof}
For $k=0$, \eqref{eq:duality_discrete} is obvious since $\Phi^\D_{0,0}(x_i)=x_i$. For $k=1$, by \eqref{eq:markov_chain} we have
\begin{align*}
&	\int_{\R^d} \EE[g(\Phi^\D_{1,0}(x))] f^\D_0(dx) 
	=\sum_{j\in\mI^\D}f^\D_0(j)\EE[ g(\Phi^\D_{1,0}(x_j))]	\\
	&=\sum_{j\in\mI^\D}f^\D_0(j)\Big(\sum_{i\in\mI^\D} g(x_i)\PP(\Phi^\D_{1,0}(x_j)=x_i)\Big)
	=\sum_{j\in\mI^\D}f^\D_0(j)\Big(\sum_{i\in\mI^\D}g(x_i)p^{(0)}_{ji}\Big)\\
	&=\sum_{i\in\mI^\D}g(x_i)\Big(\sum_{j\in\mI^\D}p^{(0)}_{ji}f^\D_0(j)\Big)	=\sum_{i\in\mI^\D}g(x_i)f^\D_{1}(i)=\int_{\R^d} g(x) f^\D_1(dx),
\end{align*}
and   \eqref{eq:duality_discrete} follows. Repeating a similar argument for $k=2,3,\dots,N$, we obtain the statement.
\end{proof}
\begin{lemma}\label{lem:reg_sol_fp}
	The scheme \eqref{eq:approx_measure} is conservative. Moreover, if $\Dx^2/h$ is uniformly  bounded for $\D\to 0$, then there exists a constant  $C$, independent of $\D$, such that for $k$, $k'\in\{0,\dots,N\}$
	\begin{align}
	   & \int_{\R^d}|x|^2f^\D_k(dx)\le C,\label{eq:2_moment}\\
	&	\mW_1(f^\D_{k'},f^\D_k)\le C(h|k'-k|)^\half.\label{eq:cont_meas}
	\end{align}
\end{lemma}

\begin{proof}
	We first prove that the scheme is conservative. Indeed, for all $k=0,\hdots, N-1$, 
	\begin{align*}
		\sum_{i\in \mI^\Dx} f^\Delta_{k+1}(i) =\sum_{i\in \mI^\Dx}\frac{1}{2r}\sum\limits_{\ell=1}^{r}\sum\limits_{j \in \mI^\Dx} f^\Delta_{k}(j)    \left[\beta_{i} ( \X^{k}_{\ell,+}(x_j))+\beta_{i}(\X^{k}_{\ell,-}(x_j))\right] \\
		=\frac{1}{2r}\sum\limits_{j \in \mI^\Dx} f^\Delta_{k}(j)\sum\limits_{\ell=1}^{r} \sum_{i\in \mI^\Dx}  \left[\beta_{i} ( \X^{k}_{\ell,+}(x_j))+\beta_{i}(\X^{k}_{\ell,-}(x_j))\right]
		=\sum_{j\in \mI^\Dx} f^\Delta_{k}(j).
	\end{align*}
	Moreover, by \eqref{eq:approx_init_meas2}, we have
	$\sum_{i\in \mI^\Dx} f^\Delta_{k}(i)=\sum_{i\in \mI^\Dx} f^\D_{0}(i)=\int_{\R^d}df_0=1$.
\par
	We now prove \eqref{eq:2_moment}. 	With a computation similar to the previous one and taking into account \eqref{eq:bound_interp_2}, we have
	\begin{align*}
		&\int_{\R^d}|x|^2f^\D_{k+1}(dx)=\sum_{i\in  \mI^\Dx}|x_i|^2f^\D_{k+1}(i)\\
		& =\frac{1}{2r}\sum\limits_{j \in \mI^\Dx} f^\Delta_{k}(j)\sum\limits_{\ell=1}^{r} \sum_{i\in \mI^\Dx}|x_i|^2  \left[\beta_{i} ( \X^{k}_{\ell,+}(x_j))+\beta_{i}(\X^{k}_{\ell,-}(x_j))\right]\\
		&=\sum\limits_{j \in \mI^\Dx}f^\Delta_{k}(j)\frac{1}{2r}\sum\limits_{\ell=1}^{r} \left(I[|\cdot|^2]( \X^{k}_{\ell,+}(x_j))+I[|\cdot|^2]( \X^{k}_{\ell,-}(x_j))\right)\\
		&\le\sum\limits_{j \in \mI^\Dx}f^\Delta_{k}(j) \frac{1}{2r}\sum\limits_{\ell=1}^{r} \left(| \X^{k}_{\ell,+}(x_j)|^2+| \X^{k}_{\ell,-}(x_j)|^2+C\Dx^2\right)\\
		&	=\sum\limits_{j \in \mI^\Dx}f^\Delta_{k}(j) \frac{1}{2r}\sum\limits_{\ell=1}^{r} 
		\left(|x_j|^2+h^2|b^h_k(x_j)|^2+2hr |\sigma^{\ell,h}_k(x_j)|^2 -2hx_j b^h_k(x_j)\right)+C\Dx^2\\
		&= \int_{\R^d}|x|^2f^\D_{k}(dx)(1+Ch)+ Ch+C\Dx^2,
	\end{align*}
	where we used condition the (OSLC) assumption in \eqref{hyp:OSLC} to estimate   $-2hx_j b^h_k(x_j)$.
	Indeed
	\begin{align*}
		-hx_j b^h_k(x_j)=-h x_j\frac{1}{h}\int_{kh}^{(k+1)h} b_h(x_j,s)ds\le
		|x_j|^2	\int_{kh}^{(k+1)h}C_1(s)ds+ x_j \int_{kh}^{(k+1)h}b_h(s,0)ds.
	\end{align*}
    
	Iterating, we get
	\begin{align*}
		\int_{\R^d}|x|^2f^\D_{k+1}(dx)&\le (1+Ch)^{k+1}\int_{\R^d}|x|^2f^\D_{0}(dx)+  (k+1)( Ch+\Dx^2)\\
		&\le e^{(k+1)h}\int_{\R^d}|x|^2f^\D_{0}(dx)+(k+1)h\left(C+\frac{\Dx^2}{h}\right).
	\end{align*}
To prove \eqref{eq:cont_meas}, given a 1-Lipschitz function $\phi$, by \eqref{eq:duality_discrete} and \eqref{eq:dp_discrete3} we have
\begin{align*}
&\int_{\Rd}g(x)(f^\D_k(dx)-f^\D_{k'}(dx))=\int\big[\EE[g(\Phi^\D_{k,0}(x))]-\EE[g(\Phi^\D_{k',0}(x))]\big]f^\D_0(dx)\\
&\le \sum_j \EE|\Phi^\D_{k,0}(x_j)-\Phi^\D_{k',0}(x_j)|f^\D_0(dx)\le \sum_j (\EE|\Phi^\D_{k,0}(x_j)-\Phi^\D_{k',0}(x_j)|^2)^\half f^\D_0(dx)\\
&\le C(h(k-k'))^\half.
\end{align*}
	
\end{proof}
We extend $f^\D$ given by \eqref{eq:approx_measure} to an element of $C([0,T], \mP_1(\R^d))$ by setting
\begin{equation}\label{eq:interp_meas}
	f^\D(t)=\frac{(k+1)h-t}{h}	f^\D_k +\frac{t-kh}{h}f^\D_{k+1}(x),\quad t\in [kh, (k+1)h),\ k=0,\dots,N-1.
\end{equation}
It is immediate that $f^\D$ satisfies the properties corresponding to \eqref{eq:2_moment} and \eqref{eq:cont_meas} for any $t\in [0,T]$. By \cite[Lemma 2.1]{cs2}  the sequence $f^\D $ converges, up to a subsequence, in the space $C([0,T);\mP_1(\R^d))$. In the next result, we identify the limit as the duality solution of \eqref{eq:FP}.
\begin{thm}\label{thm:fully_to_continuous_FP}
	For $ \D=(h,\Dx)\to 0$ with $\Dx^2/h$ uniformly bounded, the measure $f^\D$ defined in \eqref{eq:interp_meas}	converges to the duality solution $f$ of \eqref{eq:FP} in  $C([0,T);\mP_1(\R^d))$.
\end{thm}
\begin{proof} 
Consider a converging subsequence of $f^\D$ (still indexed by $\D$) and
assume that  $k_h h\to t$ for $h\to 0$. Then, given a 1-Lipschitz function  $g$, by Def. \ref{def:dual_sol} and \eqref{eq:duality_discrete},    we have that
	\begin{align*}
		&\int_{\R^d} g(x) f(t, dx) = \int_{\R^d} \EE[g(\Phi_{t,0}(x))] f_0(dx),\\
		&\int_{\R^d} g(x)df^\D_{k_h}(x)=\int_{\R^d}\EE[g(\Phi^\D_{k_h,0}(x))] f^\D_0(dx).		
	\end{align*}
Moreover $\EE[g(\Phi_{t,0}(x))]$ is the solution at time $0$ of \eqref{eq:HJ} with final datum $g(x)$ at time $t$ and $\EE[g(\Phi^\D_{k,0}(x))]$ is the solution of the scheme \eqref{eq:scheme_HJ} at step  $0$ with final datum $\hat g(i)=g(x_i)$  at step $k$. Hence, for $\Phi^h  _{k_h,0}$  defined as in \eqref{eq:Euler_chain}, we have
\begin{align*}
&\int_{\R^d} g(x)( f^\D  _{k_h}(dx)-   f(t, dx)) =\int_{\R^d}\EE[g(\Phi^\D  _{k_h,0}(x))]f^\D_0(dx)-\int_{\R^d}\EE[g(\Phi_{t,0}(x))] f_0(dx)\\
&=\int_{\R^d}\left(\EE[g(\Phi^\D  _{k_h,0}(x))-\EE[g(\Phi^h  _{k_h,0}(x))]\right)f^\D_0(dx)\\
&+\int_{\R^d}\EE[g(\Phi^h  _{k_h,0}(x))]f^\D_0(dx)-\int_{\R^d}\EE[g(\Phi_{t,0}(x))] f_0(dx),
\end{align*} 
By \eqref{eq:Lipschitz} and Prop. \ref{prop:semi-to-fully}, we have
\begin{equation}\label{eq:e_2}
	\begin{split}
	&\int_{\R^d}\left(\EE[g(\Phi^\D  _{k_h,0}(x))]-\EE[g(\Phi^h  _{k_h,0}(x))]\right)f^\D_0(dx)\\
	&	= \int_{\R^d}\left(\EE[g(\Phi^\D  _{k_h,0}(x))]-\EE[I[g(\Phi^\D  _{k_h,0}(\cdot))](x)]\right)f^\D_0(dx)\\	
	&+	 \int_{\R^d}\left(\EE[I[g(\Phi^\D  _{k_h,0}(\cdot))](x)]-\EE[g(\Phi^h  _{k_h,0}(x))]\right)f^\D_0(dx)\\
&	 \le  C \left(\Dx+\frac{\Dx}{h}\right)\int_{\R^d}f^\D_0(dx).
	\end{split}
\end{equation}
Arguing as in \eqref{eq:stima_lip}, we have that the function $\EE[g(\Phi^h  _{k_h,0}(x))]$ is Lipschitz continuous in $x$ with constant $L$, uniformly in $\D$. Hence
\begin{equation}\label{eq:e_3}
	\begin{split}
		&\int_{\R^d}\EE[g(\Phi^h  _{k_h,0}(x))]f^\D_0(dx)-\int_{\R^d}\EE[g(\Phi_{t,0}(x))] f_0(dx)\\
		&=L\int_{\R^d}\frac{1}{L}\EE[g(\Phi^h  _{k_h,0}(x))](f^\D_0(dx)- f_0(dx))\\
		&+\int_{\R^d}\Big(\EE[g(\Phi^h  _{k_h,0}(x))]-\EE[g(\Phi_{t,0}(x))]\Big)f_0(dx)\\
		&\le L\mW_1(f^\D_0,f_0)+\int_{\R^d}\left(\EE[g(\Phi^h  _{k_h,0}(x))]-\EE[g(\Phi_{t,0}(x))]\right)f_0(dx).
	\end{split}
\end{equation}
By \eqref{eq:e_2} and \eqref{eq:e_3}, we have 
\begin{equation}\label{eq:e_4}
\begin{split}
&\int_{\R^d} g(x)(  f^\D  _{k_h}(dx)-   f(t, dx))\le	C \left(\Dx+\frac{\Dx}{h}\right)\int_{\R^d}f^\D_0(dx)\\
&+ L\mW_1(f^\D_0,f_0)+\int_{\R^d}\left(\EE[g(\Phi^h  _{k_h,0}(x))]-\EE[g(\Phi_{t,0}(x))]\right)f_0(dx).
\end{split}
\end{equation}
By  \eqref{hyp:m_0} and the proof of Theorem \ref{thm:fully_to_continuous_HJ}, $\EE[g(\Phi^h  _{k_h,0}(x))]$ converges to $\EE[g(\Phi_{t,0}(x))]$  uniformly in $[0,T]\times\supp[f_0]$. Moreover, by  Lemma \ref{convf0}, $f^\D_0$ converges to $f_0$ in $\mW_1$. Therefore, passing to the limit in \eqref{eq:e_4}, we get
\[
\lim_{\D\to 0} \int_{\R^d} g(x)(  f^\D  _{k_h}(dx)-   f(t, dx))=0\quad \text{for any $g$ 1-Lipschitz.}
\]
Hence  $\mW_1(f^\D(t),f(t))\lra 0$ for $\D\to 0$ for any $t\in [0,T]$.
Since any converging subsequence of $f^\D$ converges to the duality solution $f$ of \eqref{eq:FP},   we conclude that all the sequence $f^\D$ converges fo $f$.
	\end{proof}
\section{Numerical illustration}
In this section, we provide some numerical example to illustrate the properties of the scheme.
\begin{figure}[ht]
\centering
\includegraphics[width=0.8\textwidth]{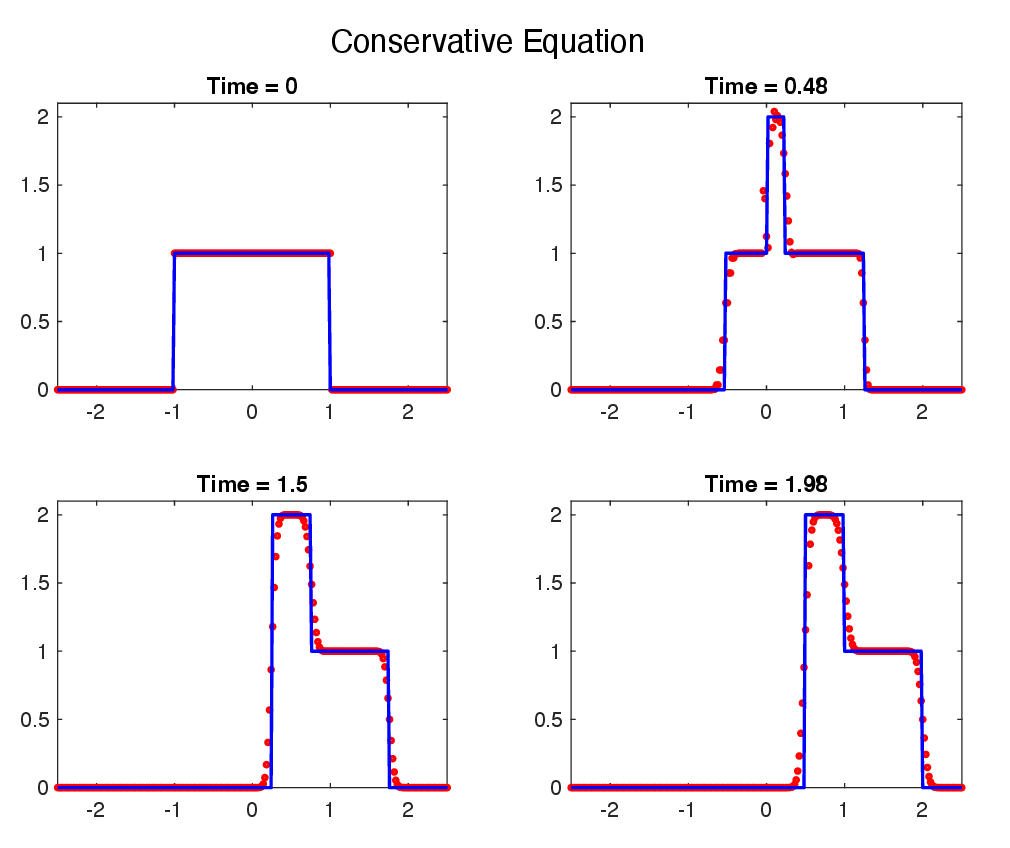}
\caption{Test 1. (Forward) Conservative equation. The discretization parameters are \(\Delta = (\Delta x, h) = (0.02, 0.06)\).}
\label{F:2}
\end{figure}
\begin{figure}[ht]
\centering
\includegraphics[width=0.8\textwidth]{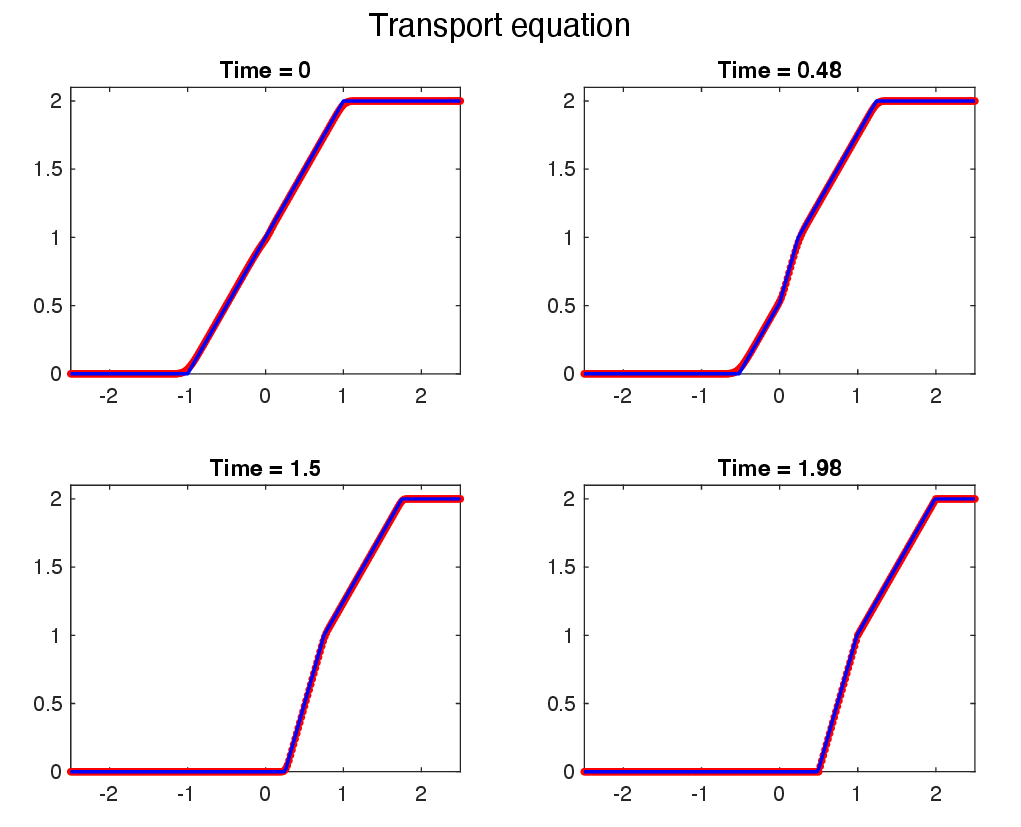}
\caption{Test 1. (Backward) Inviscid transport equation. The initial solution is chosen as the cumulative distribution function of the exact solution for the conservative case. Comparison between exact and numerical solutions at various times. Discretization parameters: \(\Delta = (\Delta x, h) = (0.02, 0.06)\).}
\label{F:1}
\end{figure}

\subsection*{Test 1}

We consider the   case of a first order one-dimensional problem for which we can compute  the analytical solutions, allowing direct comparison with the numerical approximations. This example is taken by \cite{dlv}, where a similar analysis for an upwind scheme   is carried out. 
We study equations \eqref{eq:HJ} and \eqref{eq:FP_1} with \( a(t,x) \equiv 0 \) and
\begin{equation*}
  -b(t,x) = \begin{cases}
1& \text{for } x < 0, \\
\frac{1}{2} & \text{for } x \geq 0;
\end{cases}
\end{equation*}
since the velocity field is $-b$ in its forward formulation (cfr. Remark \ref{rem:first_order}). Note that such jump in the velocity field verify the assumption (OSLC).
Coefficients of type arise  pretty naturally in supply chain models and similar problems involving (discontinuous) congested transportation flows \cite{FGM,fgu}.
We choose as initial condition for the conservative equation \( f_0(x) = \mathbbm{1}_{[-1,1]}(x) \). Then, the analytical solution   is given by 
\begin{equation*}
f(t,x) = 
\begin{cases}
\mathbbm{1}_{[-1+t,0)}(x)+ 2 \cdot \mathbbm{1}_{[0,t/2)}(x) + \mathbbm{1}_{[t/2,1+t/2)}(x) & \text{for } t \leq 1, \\
2 \cdot \mathbbm{1}_{[1/2(t-1),t/2)}(x) + \mathbbm{1}_{[t/2,1+t/2]}(x) & \text{for } t > 1.
\end{cases}
\end{equation*}
We observe that the   scheme  is highly effective to approximate the compression wave, even with pretty high discretization parameters.

\begin{figure}[ht]
\centering
\includegraphics[width=0.5\textwidth]{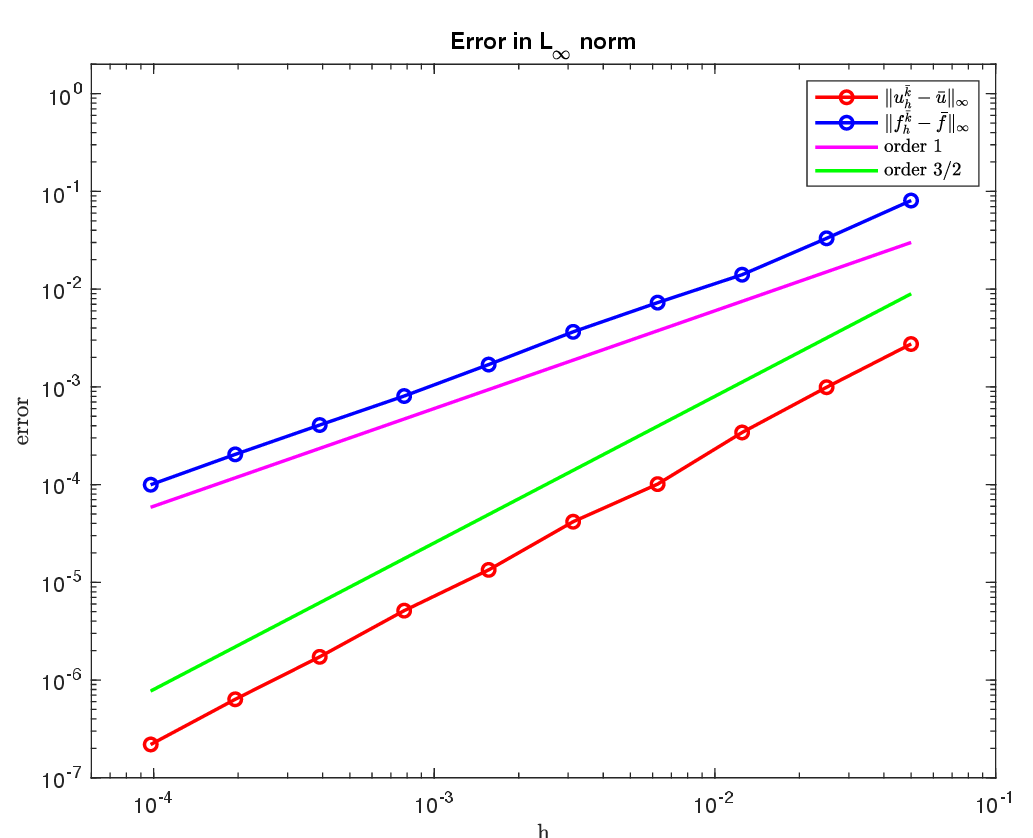}
\hspace{-0.5cm}
\includegraphics[width=0.5\textwidth]{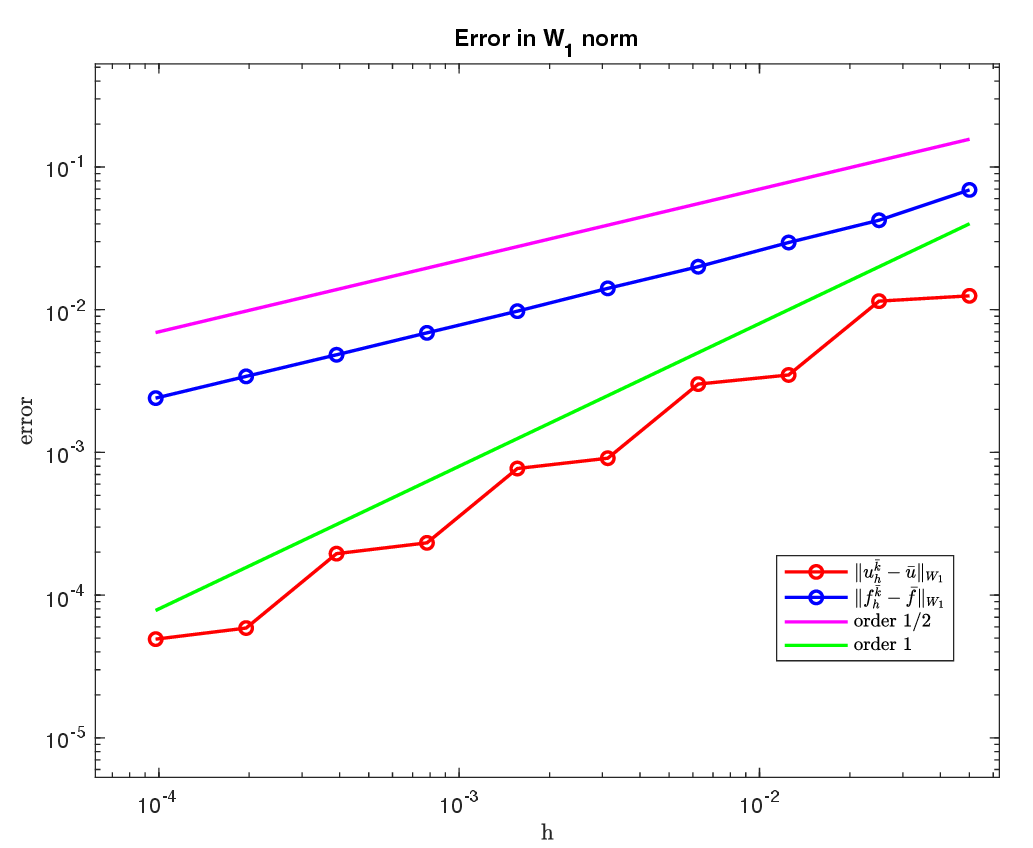}
\caption{Test 1. Comparison of convergence rates in \(L^\infty\) and Wasserstein norms for the two numerical schemes.}
\label{F:3}
\end{figure}
Figure \ref{F:2} shows the numerical results with discretization parameters \(  (\Delta x, h) = (0.02, 0.06)\). Note that no restrictive condition on the choice of \(h\) is required. Since this is the purely inviscid case (\(a \equiv 0\)), the scheme correctly captures the expected solution, without any oscillation around the discontinuity of \(b\). We highlight that here and in the following, to avoid some boundary effect not considered in this study, the approximation is performed in a numerical domain sufficiently big (in this case the interval $[-5,5]$). We stress also on the role of the mollifier $\rho:\R\rightarrow\R$ which, scaled by the time step $\rho_h(x)=h^{-1}\rho(x/h)$, provide a regularization of the discontinuous velocity field that vanishes when $h\rightarrow 0^+$. In this  test we take  
$$ \rho(x):=\frac{1}{\sqrt{2\pi}}e^{-\frac{x^2}{2}}.$$
Exploiting the duality between \eqref{eq:HJ} and \eqref{eq:FP}, we consider as final condition for the transport equation \eqref{eq:HJ} the cumulative distribution of \(f(2,x)\), obtaining
\begin{equation*}
u_T(x) =
\begin{cases}
0 & \text{for } x < 1, \\
2x-\frac{1}{2}  & \text{for } 1/2 \leq x < 1, \\
x  & \text{for } 1 \leq x < 2, \\
2 & \text{for } x \geq 2.
\end{cases}
\end{equation*}

The corresponding analytical solution for \(t \in [0,2]\) is given by
\begin{equation*}
u(t,x) =
\begin{cases}
(x+1-t)\mathbbm{1}_{[-1+t,0)}(x) + (2x+1-t) \cdot \mathbbm{1}_{[0,\frac{t}{2})}(x) + (x-\frac{t}{2}+1)\mathbbm{1}_{[\frac{t}{2},1+\frac{t}{2})}(x) \\
\hspace{8cm}+2\,\mathbbm{1}_{[1+\frac{t}{2},+\infty)}(x), \hspace{0.6cm} \text{for } t \leq 1, \\
(2\,x+1-t)\,\mathbbm{1}_{[\frac{1}{2}(t-1),\frac{t}{2})}(x) + (x-\frac{t}{2}+1) \mathbbm{1}_{[\frac{t}{2},1+\frac{t}{2}}(x) +2\,\mathbbm{1}_{[1+\frac{t}{2},+\infty)}(x),\\
\hspace{11.3cm}  \text{for } t > 1.
\end{cases}
\end{equation*}
Essentially, we use the backward transport equation to recover the initial state of the system. Thus, we expect the approximate solution to satisfy the relation \(u_x(t,x) \approx f(t,x)\), which holds exactly for the analytical solution under some regularity assumption. As shown in Figure \ref{F:1}, the numerical approximation matches the analytical solution very closely.

The good performance of the scheme is further confirmed in Figure \ref{F:3}, where the convergence rates are reported both in \(L^\infty\)-norm and in Wasserstein distance 
\[
\|\phi\|_{L^\infty} = \max_{i} |\phi_i|, \quad
\mW_1(\mu, \nu) = \int_{-\infty}^{+\infty} |F_\mu(x) - F_\nu(x)| \, dx,
\]
where \(F_\mu\) and \(F_\nu\) are the cumulative distribution functions of \(\mu\) and \(\nu\), respectively.

\subsection*{Test 2}

\begin{figure}[ht]
\begin{center}
\begin{tabular}{c}
\includegraphics[width=0.9\textwidth]{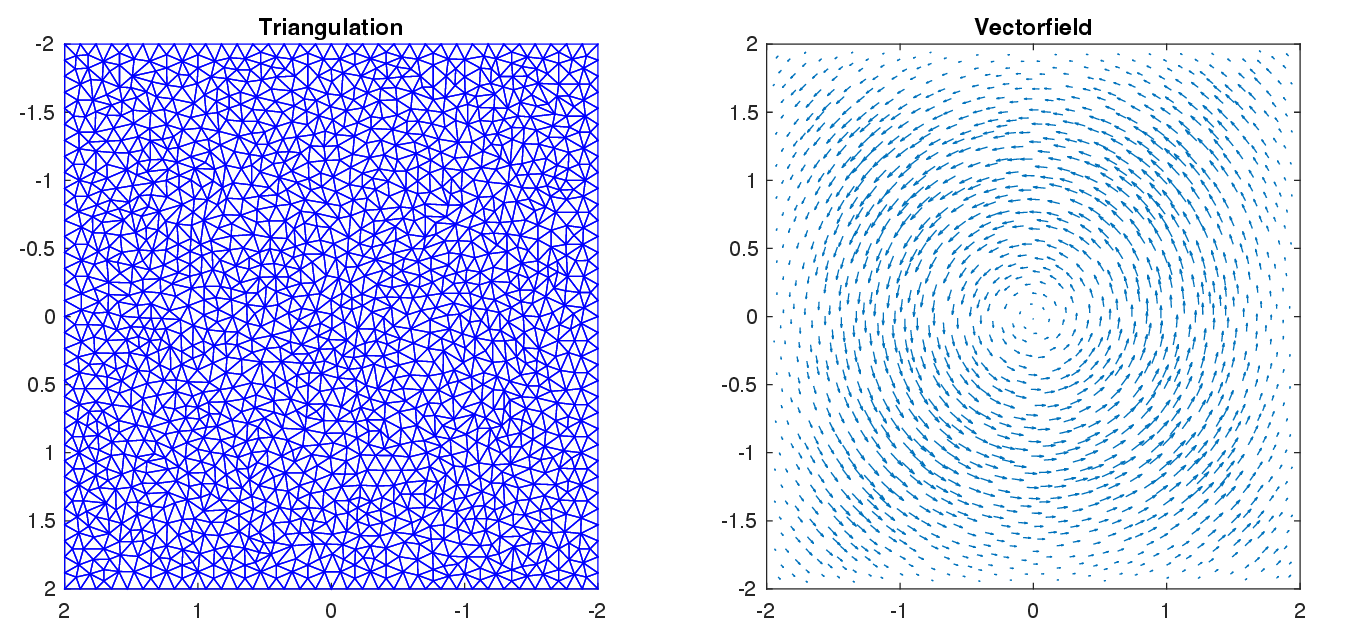}
\end{tabular}
\end{center}
\caption{Test 2. Triangulation used ($\Delta x=0.08$) and vector field $b$.}
\label{F:4.0}
\end{figure}

\begin{figure}[ht]
\begin{center}
\begin{tabular}{c}
\includegraphics[width=0.9\textwidth]{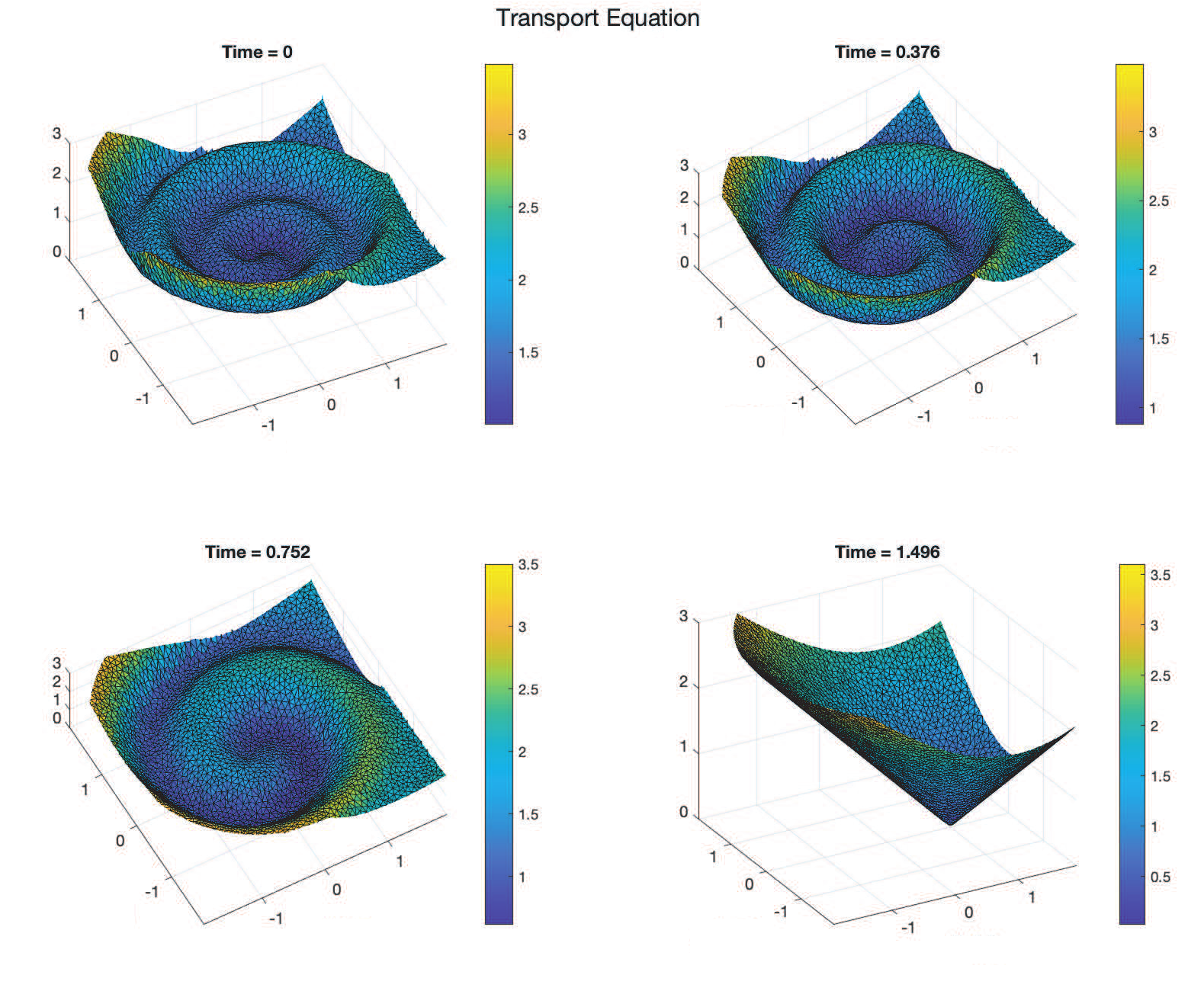}
\end{tabular}
\end{center}
\caption{Test 2. Solution of the transport equation at various moments of its evolution.}
\label{F:4}
\end{figure}

\begin{figure}[h]
\begin{center}
\begin{tabular}{c}
\includegraphics[width=0.9\textwidth]{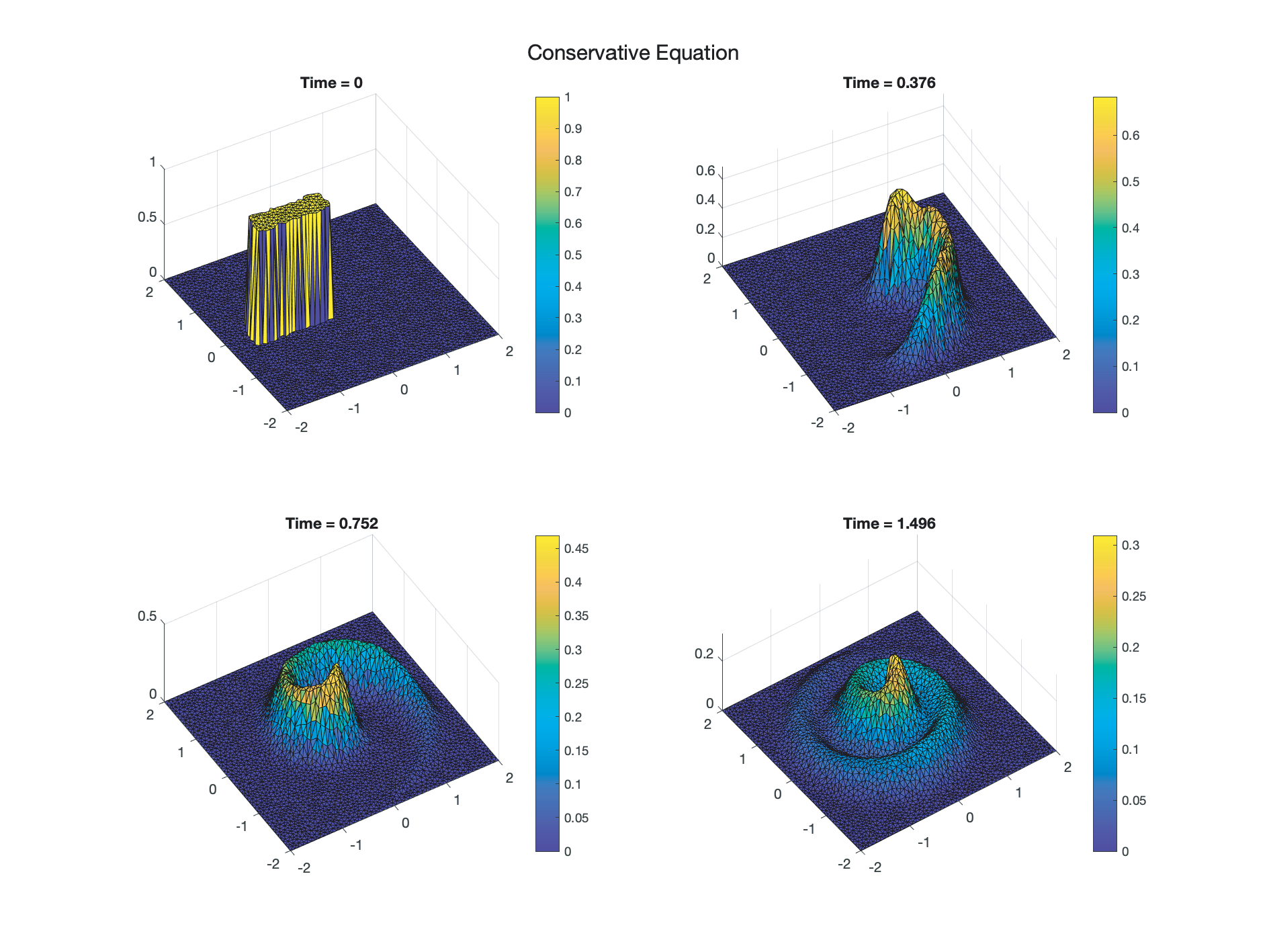}
\end{tabular}
\end{center}
\caption{Test 2. Solution of the conservation equation at various moments of its evolution.}
\label{F:5}
\end{figure}

\begin{figure}[ht]
\begin{center}
\begin{tabular}{c}
\includegraphics[width=0.9\textwidth]{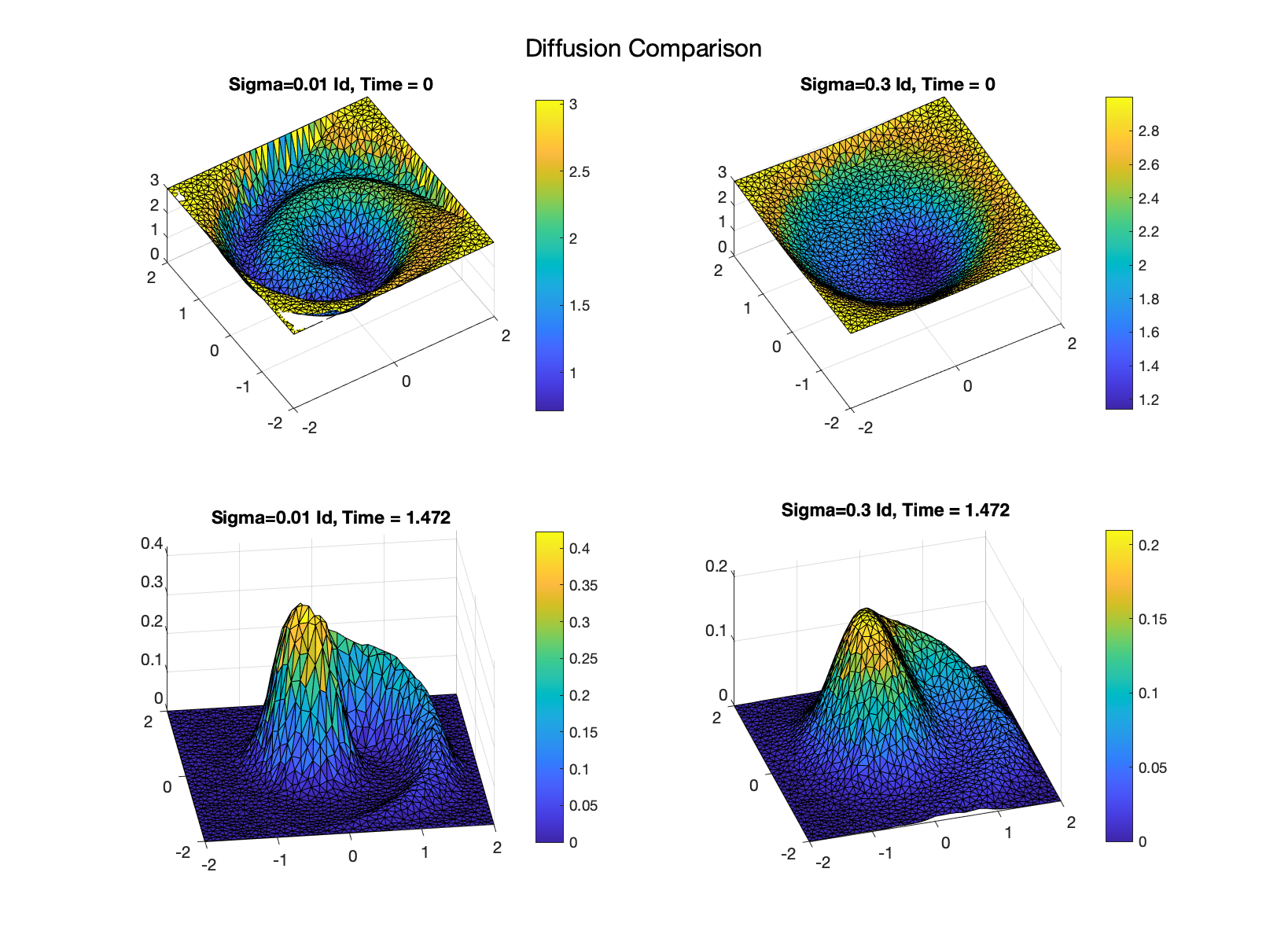}
\end{tabular}
\end{center}
\caption{Test 2. Comparison for various choices of $\sigma$: respectively,  $\sigma = 10^{-2}\text{Id}_2$ (left) and $\sigma = 0.3\,\text{Id}_2$ (right). }
\label{F:6}
\end{figure}

In the second test, we move to a two-dimensional case. The aim is to study the properties of the scheme with respect to a triangular grid. We consider a problem set in $[-2,2]^2 \subset \mathbb{R}^2$, with a continuous drift
\begin{equation*}
-b(t,x) = (2-\max(|x_1|,|x_2|))\,(x_2, -x_1),
\end{equation*}
a rotating velocity field that degenerates  near the origin and along the boundary of $[-2,2]^2$ (see Figure~\ref{F:4.0}). For this domain, we employ a triangular mesh generator with $\Delta x=0.08$ (corresponding to the maximum triangle area). The mesh is generated using the well-known software \texttt{Triangle}~\cite{Sh}, which produces a quality, conforming Delaunay triangulation of bounded 2D domains.

The other problem parameters are set as follows: $h = 2\Delta x$ and $a(t,x) \equiv \frac{1}{2}\sigma\sigma^T$, with $\sigma = 10^{-3}\,\text{Id}_2$, where $\text{Id}_2$ is the identity matrix of dimension 2. The time interval is $[0,T] = [0,1.5]$. 
In this case, we take as mollifier the standard bivariate distribution
$$ \rho(x):=\frac{1}{2\pi} \exp\left( -\frac{1}{2}(x_1^2 + x_2^2) \right).$$
We begin by considering the backward viscous transport equation with final condition
\[
u_T(x) = \sqrt{(x_1-1)^2 + x_2^2}.
\]
Figure~\ref{F:4} shows the approximate solution at different times during its evolution. In this case, the presence of a nondifferentiable point in the final condition (only Lipschitz continuity is required by Theorem~\ref{thm:fully_to_continuous_HJ}) does not affect the performance of the numerical scheme, which successfully approximates the solution without introducing spurious oscillations.

We now turn to analyzing the same choice of parameters for the viscous conservation equation, with initial condition
\[
f_0(x) = \mathbbm{1}_{[-1.5,-0.1] \times [-0.25,0.25]}(x),
\]
where the initial data is clearly discontinuous. Figure~\ref{F:5} displays the approximate solution at various times during its evolution. It is possible to observe that a higher density remains concentrated near the center of the domain (where the vector field $b$ is smaller), while the rest of the density moves more rapidly in a non-uniform circular motion.

In Figure \ref{F:6} we observe how various choices of the diffusion coefficient $\sigma$ influences the solution. We report here the previous test with, respectively,  $\sigma = 10^{-2}\text{Id}_2$ (left) and $\sigma = 0.3\,\text{Id}_2$ (right).

\begin{figure}[ht]
\begin{center}
\begin{tabular}{c}
\includegraphics[width=0.9\textwidth]{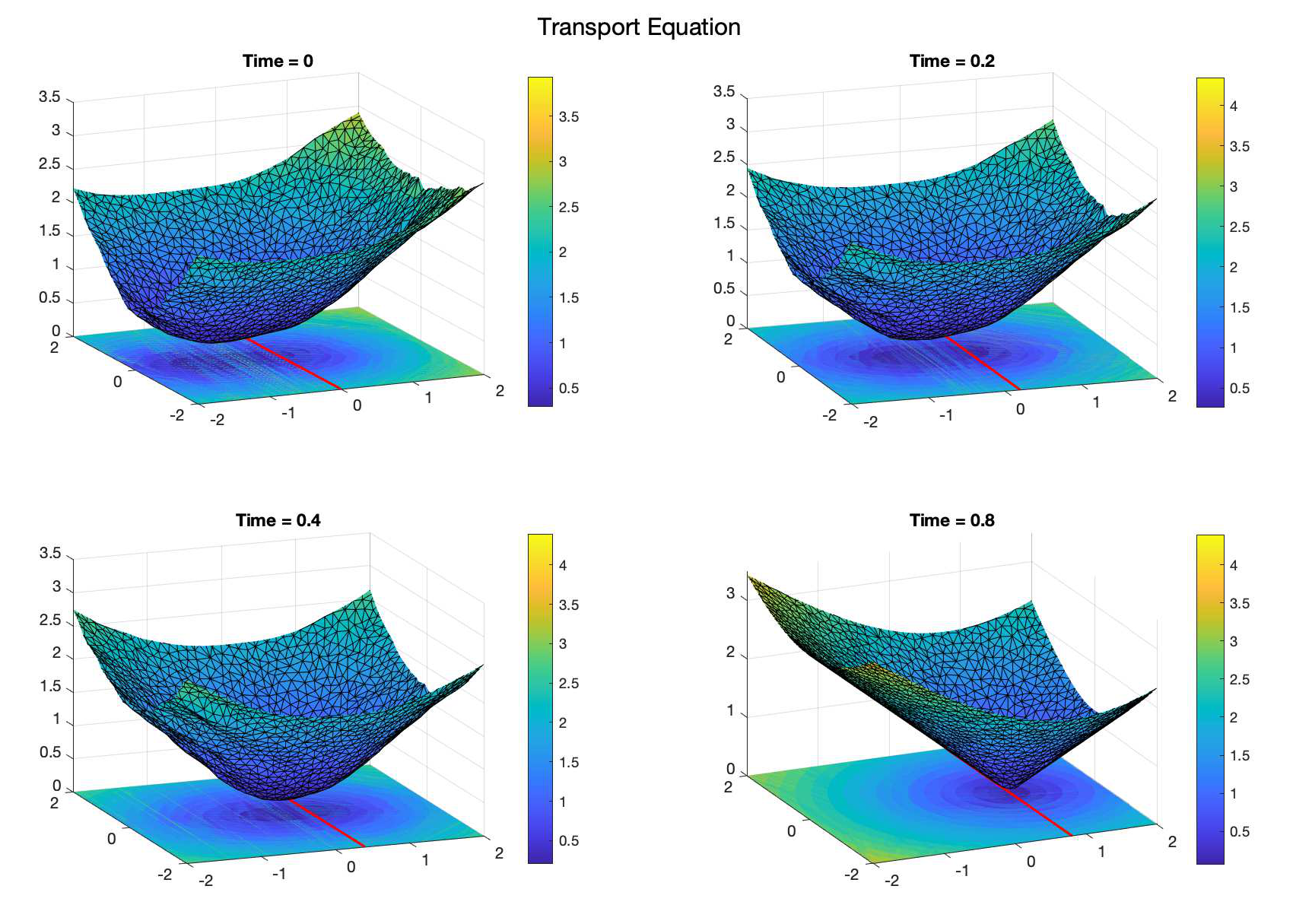}
\end{tabular}
\end{center}
\caption{Test 3. Transport equation. The red line marks the position, varying in time, of the discontinuity in the velocity field. The contour lines are also displayed for a better interpretation of the plot.  }\label{F:7}
\end{figure}

\begin{figure}[ht]
\begin{center}
\begin{tabular}{c}
\includegraphics[width=0.9\textwidth]{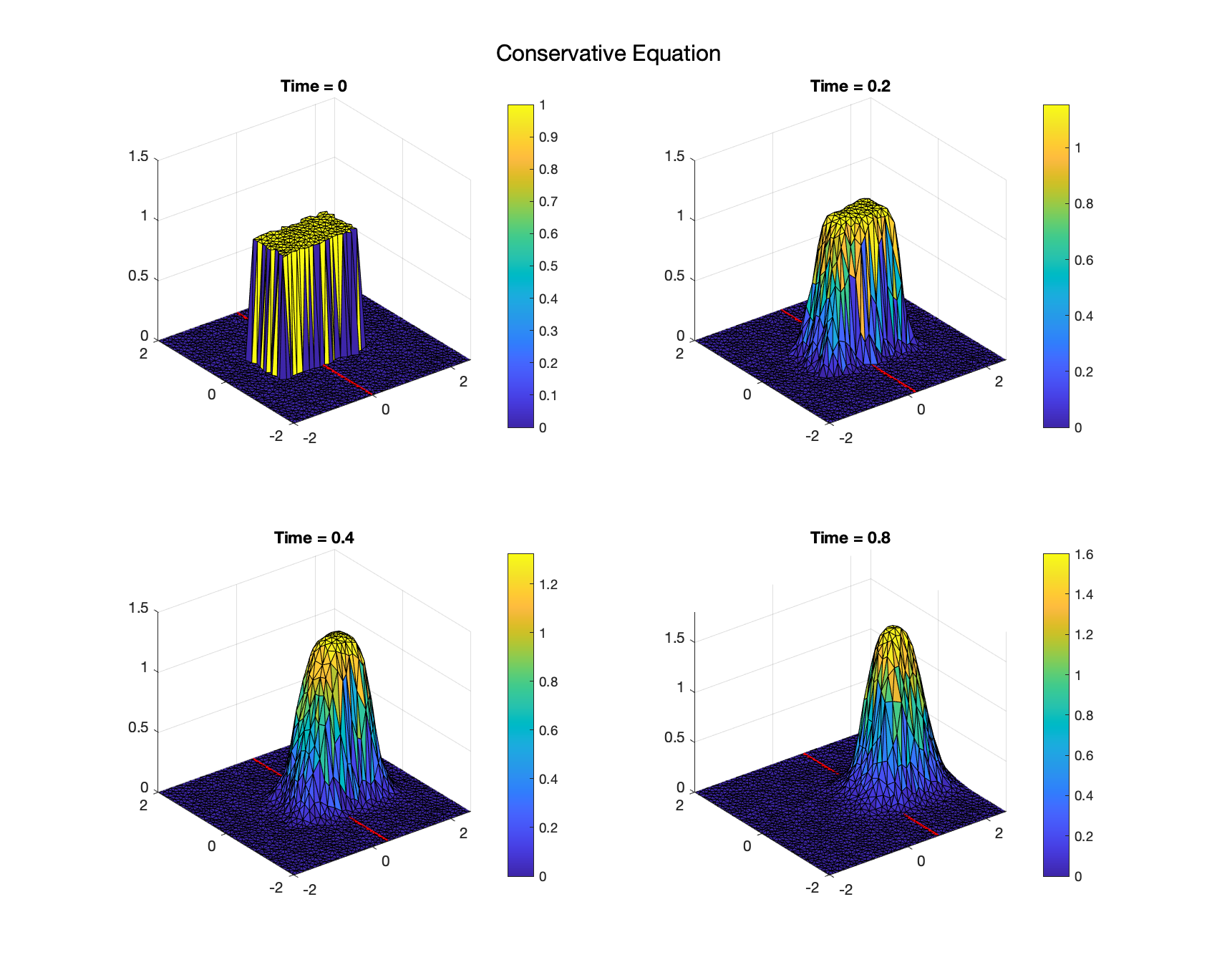}
\end{tabular}
\end{center}
\caption{Test 3. Conservation equation. The red line marks the position, varying in time, of the discontinuity in the velocity field. The initial density concentrates while passing through the moving discontinuity.}\label{F:8}
\end{figure}

\subsection*{Test 3}

In our third test we pass to the case of a discontinuous velocity field $b$, time-varying, that verifies the (OSLC) condition. We consider

{\[-b(t,x,y)=\begin{cases}
	(3/2,0)\quad& x<t\\
	(1/2,0)\quad& x\ge t
\end{cases}\]
where we can notice that the discontinuity moves in the same direction of the motion with unitary speed, while beyond the discontinuity, since the speed is lower, we expect then that a compression wave will appear. In this test the diffusion coefficient is variable, while   the scenario remains convection dominant. We set
 \[\sigma(t,x)=0.1 \,|\cos(\pi\,x_1)\cos(\pi\,x_2)|\,\text{Id}_2\]
 where we observe that such function degenerate to zero in some points of the domain.

In Figure \ref{F:7} we observe the (backward) evolution of the viscous transport equation in the time interval $[0,0.8]$, starting from the final condition

\[
u_T(x) = \sqrt{(x_1-0.8)^2 + x_2^2}.
\]
Here the discretization parameters are set as $(\Delta x,h)=(0.01,0.02)$.
In the figure, the red line marks the position of the discontinuity in the velocity field $b$. Such discontinuity is visible in the gradient of the approximated solution, while the diffusion tends to smooth up the minimum point of the initial solution. Also in this case, to avoid boundary effects, the problem is computed in a much larger domain (in particular $[-4,4]^2$).

We now turn to analyzing the same choice of parameters for the viscous conservation equation, with initial condition
\[
f_0(x) = \mathbbm{1}_{[-1,1] \times [-0.5,0.5]}(x).
\]
 Figure~\ref{F:8} shows the approximate solution at several time points during its evolution. A compression wave can be observed, as indicated by the fact that the peak concentration at the final time exceeds the maximum of the initial condition. This concentration effect is smoothed by diffusion and by the fact that the discontinuity propagates in the same direction as the drift.

\section{Conclusions}

In this work, we developed and analyzed semi-Lagrangian schemes on unstructured meshes for viscous transport and conservative equations with coefficients satisfying a one-sided Lipschitz condition. By leveraging a probabilistic framework and the dual characterization of solutions as viscosity and duality solutions, we proved convergence under minimal regularity assumptions. The schemes exhibit strong robustness in the presence of irregular data and extend classical semi-Lagrangian methods beyond the Di Perna-Lions setting.

Numerical results confirm convergence to the correct solution and validate the method’s accuracy and stability on unstructured grids, making it particularly effective for high-dimensional and geometrically complex domains. The approach is especially relevant for Mean Field Games, where monotone interactions and degenerate dynamics frequently arise. Its compatibility with measure-theoretic formulations and probabilistic interpretations enhances its suitability for modeling socio-economic and collective behavior dynamics.

Future directions include extending the method to fully nonlinear second-order equations and systems with additional couplings or controls.



\end{document}